\documentclass{article}

 \usepackage{graphicx}

\usepackage{amsmath, bm,amssymb, amsthm, enumerate}

\usepackage{amsmath,amssymb,amsthm,enumerate,mathrsfs}
\newtheorem{theorem}{Theorem}[section]
\newtheorem{corollary}[theorem]{Corollary}
\newtheorem{lemma}[theorem]{Lemma}
\newtheorem{proposition}[theorem]{Proposition}
\theoremstyle{definition}
\newtheorem{definition}[theorem]{Definition}
\newtheorem{exam}[theorem]{Example}
\newtheorem{assumption}{Assumption}[section]

\newtheorem{remark}[theorem]{Remark}
\newtheorem{condition}{Condition}[section]

\def\[{\begin{equation}}
\def\]{\end{equation}}

\def\nn{\nonumber}
\def\t{\top}

\def\C{{\mathscr C}}

\numberwithin{equation}{section}

\begin{document}
\makeatletter

\begin{center}
\large{\bf Properties of the solution set of generalized polynomial complementarity problems}
\end{center}\vspace{5mm}
\begin{center}

\textsc{Liyun Ling,\quad Chen Ling \;\; and\;\; Hongjin He}\end{center}

\vspace{2mm}

\footnotesize{
\noindent\begin{minipage}{14cm}
{\bf Abstract:}
In this paper, we consider the {\it generalized polynomial complementarity problem} (GPCP), which covers the recently introduced {\it polynomial complementarity problem} (PCP) and the well studied {\it tensor complementarity problem} (TCP) as special cases. By exploiting the structure of tensors, we first show that the solution set of GPCPs is nonempty and compact when a pair of leading tensors is cone {\bf ER}. Then, we study some topological properties of the solution set of GPCPs under the condition that the leading tensor pair is cone ${\bf R}_0$. Finally, we study a notable global Lipschitzian error bound of the solution set of GPCPs, which is better than the results obtained in the current PCPs and TCPs literature. Moreover, such an error bound is potentially helpful for finding and analyzing numerical solutions to the problem under consideration.
\end{minipage}
 \\[5mm]


\hbox to14cm{\hrulefill}\par

\section{Introduction}\label{Introd}
Let $F,G: \mathbb{R}^n\rightarrow \mathbb{R}^n$ be two given continuous functions and $K\subset\mathbb{R}^n$ be a closed convex cone. The classical {\it generalized complementarity problem} (GCP) is to find a vector ${\bm x}\in \mathbb{R}^n$ such that
\[\label{GCP}
F({\bm x})\in K,~~G({\bm x})\in K^\ast,~~{\rm and}~~\langle F({\bm x}),G({\bm x})\rangle=0,
\]
where $\langle\cdot,\cdot\rangle$ denotes the standard inner product in real Euclidean space, and $K^*$ is the dual cone of $K$, which is defined by
$$K^*:=\left\{\bm y\in \mathbb{R}^{n}~|~\langle \bm y,\bm x\rangle\geq 0,~\forall ~\bm x\in K\right \}.$$
It is easy to see that $K^*$ is also a closed convex cone in $\mathbb{R}^n$. In what follows, we denote \eqref{GCP} by ${\rm GCP}(F,G,K)$ for simplicity. Clearly, ${\rm GCP}(F,G,K)$ has a unified form including the typical {\it nonlinear complementarity problem} (NCP, $K:={\mathbb R}^n_+$ and $G({\bm x}):={\bm x}$ in \eqref{GCP}) as a special case, which has a large amount of applications in engineering, economics, finance, and robust optimization, e.g., see \cite{FP03,FP97,Wal13,Zha07} and the references therein.

In the past decade, it has been well documented that tensor, which is a natural extension of matrix, is a powerful tool for the treatment of many real-world problems, such as the biomedical data processing \cite{BACGWM15,BV08,QYW10,ZZZXC16}, higher-order Markov chains \cite{NQZ09}, $n$-person noncooperative games (also named as multilinear games) \cite{HQ16}, best-rank one approximation in data analysis \cite{QWW09}, hypergraphs \cite{CCLQ18}, to name just a few. Therefore, the complementarity problem with tensors, which is called {\it tensor complementarity problem} (TCP) in the seminal works \cite{SQ15b,SQ16}, also have attracted a lot of considerable attention in optimization community, e.g., see \cite{BHW16,CQW16,DQW13,GLQX15,SQ15a,SQ14,WHB16} and the references therein. For given positive integers $m$ and $n$, we call $\mathcal{A}= (a_{j_1j_2\ldots j_m} )$, where $a_{j_1j_2\ldots j_m} \in \mathbb{R}$ for $1\leq j_1, j_2,\ldots, j_m\leq n$, an $m$-th order $n$-dimensional real square tensor and denote by $\mathbb{T}_{m,n}$ the space of $m$-th order $n$-dimensional real square tensors. For a tensor $\mathcal{A}=(a_{j_1j_2\ldots j_m})\in \mathbb{T}_{m,n}$ and a vector $\bm x=(x_1,x_2,\ldots,x_n)^\top\in \mathbb{R}^n$, we define $\mathcal{A}{\bm x}^{m-1}$ being a vector whose $j$-th component is given by
$$
(\mathcal{A}{\bm x}^{m-1})_j=\sum_{j_2,\ldots,j_m=1}^na_{jj_2\ldots j_m}x_{j_2}\cdots x_{j_m},\quad j=1,2,\ldots,n.
$$
With the above notation, the concrete form of TCP is specified as $F(\bm x):=\mathcal{A}\bm x^{m-1}+\bm q$ with a vector $\bm q\in \mathbb{R}^n$, $G({\bm x}):={\bm x}$ and $K:={\mathbb R}^n_+$ in \eqref{GCP}, denoted by ${\rm TCP}(\mathcal{A},\bm q)$. When the order of ${\mathcal A}$ is $m=2$, ${\rm TCP}(\mathcal{A},\bm q)$ immediately reduces to the deeply developed {\it linear complementarity problem} (LCP). Here, we only refer to the monograph \cite{CPS92} for more information on LCPs.

It is well known that every polynomial $H:\mathbb{R}^n\rightarrow \mathbb{R}^n$ of degree $(s-1)$ can be expressed in the form of
$$
H(\bm x):=\sum_{k=1}^{s-1}\mathcal{C}^{(k)}\bm x^{s-k}+\bm q,
$$
where $\mathcal{C}^{(k)}\in \mathbb{T}_{s-k+1,n}$ for $k=1,2,\ldots,s-1$, and $\bm q\in \mathbb{R}^n$. In this situation, we say that the polynomial $H$ is defined by $(\Delta,\bm q):=\left({\mathcal C}^{(1)},\ldots,{\mathcal C}^{(s-1)},\bm q\right)$. Most recently, Gowda \cite{G16} introduced an interesting model named by {\it polynomial complementarity problem} (PCP), which is a specialized NCP with $F$ being a polynomial defined by $(\Lambda, \bm a)$, where $\Lambda:=\left({\mathcal A}^{(1)},\ldots,{\mathcal A}^{(m-1)}\right)\in \mathcal{F}_{m,n}:=\mathbb{T}_{m,n}\times\cdots\times\mathbb{T}_{2,n}$ and ${\bm a}\in{\mathbb R}^n$. As mentioned in \cite{G16},  the PCP appears in polynomial optimization (where a real valued polynomial function is optimized over a constraint set defined by polynomials) and includes the well studied TCPs as its special case. By fully exploiting the polynomial nature of PCPs, some interesting specialized results, which seem better than the results obtained by applying the theory of NCPs to PCPs directly, have been developed, e.g., see \cite{G16,LHL17,YLH17}.

In this paper, we consider an extension of PCPs, namely, the {\it generalized polynomial complementarity problem} (GPCP), which is a special kind of GCPs with two polynomials $F(\bm x)$ and $G(\bm x)$. More specifically, both $F(\bm x)$ and $G(\bm x)$ in \eqref{GCP} are defined by $(\Lambda,{\bm a})$ and $(\Theta,{\bm b})$, respectively, where
$$
\left\{\begin{array}{ll}
\Lambda:=\left({\mathcal A}^{(1)},\ldots,{\mathcal A}^{(m-1)}\right)\in \mathcal{F}_{m,n}:=\mathbb{T}_{m,n}\times\cdots\times\mathbb{T}_{2,n}\;,&\;\; {\bm a}\in{\mathbb R}^n; \\
\Theta:=\left({\mathcal B}^{(1)},\ldots,{\mathcal B}^{(l-1)}\right)\in \mathcal{F}_{l,n}:=\mathbb{T}_{l,n}\times \cdots\times\mathbb{T}_{2,n}\;,&\;\; {\bm b}\in\mathbb{R}^n,
\end{array}\right.$$
that is,
\[\label{gpcp}F(\bm x):=\sum_{k=1}^{m-1}{\mathcal A}^{(k)}{\bm x}^{m-k}+{\bm a}\quad {\rm and}\quad G(\bm x):=\sum_{p=1}^{l-1}{\mathcal B}^{(p)}{\bm x}^{l-p}+{\bm b}.\]
In what follows, we use ${\rm GPCP}(\Lambda,\bm a,\Theta,\bm b,K)$ to represent the specialized \eqref{GCP} with $F(\bm x)$ and $G(\bm x)$ given in \eqref{gpcp}. Although the entire theory of GCPs is applicable to GPCPs, one may be interested in the specialized structure-exploiting results. In fact, the appearance of polynomials $F(\bm x)$ and $G(\bm x)$ makes the problem under consideration more challenging than the ones in PCPs and TCPs. Hence, by making use of the polynomial nature (structure of tensors) of \eqref{gpcp}, we have a threefold contribution on the problem under consideration. (i) To the best of our knowledge, this is the first work on GPCPs. Therefore, we give an affirmative answer showing that the solution set of ${\rm GPCP}(\Lambda,\bm a,\Theta,\bm b,K)$ is nonempty and compact when a pair of leading tensors is cone {\bf ER}. (ii) Under the condition that a pair of leading tensors is cone ${\bf R}_0$, we obtain some topological properties on the solution set of ${\rm GPCP}(\Lambda,\bm a,\Theta,\bm b,K)$. (iii) We prove a remarkable global Lipschitzian error bound of the solution set of ${\rm GPCP}(\Lambda,\bm a,\Theta,\bm b,K)$. Such a result has an important role in the study of unbounded asymptotics and provides valuable quantitative information about the iterates obtained at the termination of iterative algorithms for computing solutions of GPCPs. Moreover, it is noteworthy that our global error bound is better in the sense that the assumptions of our result are weaker than the ones assumed in the current PCPs and TCPs literature.

The rest of the paper is organized as follows. In Section \ref{SecPrelim}, we recall some definitions
and basic facts that will be used in the subsequent analysis. In Section \ref{SecNCS}, we obtain a sufficient condition relying on a cone ${\bf ER}$-tensor pair condition (see (\ref{equation1.2})) to guarantee the nonemptiness and compactness of the solution set of ${\rm GPCP}(\Lambda,\bm a,\Theta,\bm b,K)$. Note that such a condition reduces to the condition of the ${\bf ER}$-tensor in the case of TCPs. In Section \ref{Topprop}, we study some more topological properties of the solution set of ${\rm GPCP}(\Lambda,\bm a,\Theta,\bm b,K)$. In Section \ref{SecEA}, we present a notable global Lipschitzian error bound result for ${\rm GPCP}(\Lambda,\bm a,\Theta,\bm b, K)$ with $K:=\mathbb{R}_+^n$ under appropriate conditions. Finally, we complete this paper with drawing some concluding remarks in Section \ref{FinRem}.

\medskip
{\bf Notation}. Throughout this paper, for given positive integer $n\geq 2$, let $\mathbb{R}^n$ denote the real Euclidean space of column vectors of length $n$, and let $[n]$ be the index set $\{1, 2,\ldots, n\}$. The small letters $x,y,u,\ldots,$ and small bold letters ${\bm x}, {\bm y}, {\bm u},\ldots$ represent scalars and vectors, respectively. In particular, $\bm 0$ is the column vector in $\mathbb{R}^n$, whose all entries are zeros. For a given vector $\bm x=(x_1,x_2,\ldots,x_n)^\top\in \mathbb{R}^n$, $\|\bm x\|$ denotes the Euclidean norm of $\bm x$, and ${\bm x}_+$ denotes the orthogonal projection of $\bm x$ on $\mathbb{R}_+^n$, that is, $({\bm x}_+)_i:={\rm max}\{x_i,0\}$ for $i\in[n]$. For a given tensor ${\mathcal A}=(a_{j_1j_2\ldots j_m})\in{\mathbb T}_{m,n}$, the (squared) Frobenius norm of ${\mathcal A}$ is defined by
$$\|{\mathcal A}\|_{\rm F}:=\sqrt{\sum_{j_1,j_2,\ldots,j_m=1}^n a_{j_1j_2\ldots j_m}^2}.$$
In addition, denote by $\mathcal{I}=(\delta_{j_1j_2\cdots j_m})$ the unit tensor, where $\delta_{j_1j_2\cdots j_m}$ is the Kronecker symbol
$$
\delta_{j_1j_2\cdots j_m}=\left\{
\begin{array}{ll}
1,&\;\;{\rm if~}j_1=j_2\cdots =j_m,\\
0,&\;\;{\rm otherwise},
\end{array}
\right.
$$
In particular, we denote by $I$ the identity matrix when $m=2$. Given a tensor $\mathcal{A}=(a_{j_1j_2\ldots j_m})\in \mathbb{T}_{m,n}$ and a vector $\bm x\in \mathbb{R}^n$, we define $\mathcal{A}{\bm x}^m=\langle{\bm x}, \mathcal{A}\bm x^{m-1}\rangle$ being the value at $\bm x$ of a homogeneous polynomial.
If $F:\mathbb{R}^n\rightarrow\mathbb{R}^n$ is a vector valued function, its $j$-th component function is denoted by $F_j$.

\section{Preliminaries}\label{SecPrelim}
In this section, we introduce some basic definitions and properties that will be used in the sequels.

For a given closed convex cone $K$ in $\mathbb{R}^n$, it is well known from \cite{KS80} that the projection operator onto $K$, denoted by $P_K$, is well-defined for every $\bm x\in \mathbb{R}^n$. Moreover, we know that $P_K(\bm x)$ is the unique element in $K$ such that $\langle P_K(\bm x)-\bm x, P_K(\bm x)\rangle=0$ and $\langle P_K(\bm x)-\bm x, \bm y\rangle\geq 0$ for all $\bm y\in K$.

We now recall the concept of {\it exceptional family of elements} for a pair of functions with respect to a given closed convex cone, which plays an important role in the existence analysis of solutions of ${\rm NCPs}$ and GCPs, e.g., see \cite{IBK97,IC99,KI02}.

\begin{definition}
Let $F,G:\mathbb{R}^n\rightarrow \mathbb{R}^n$ be two given continuous functions. A set of points $\{{\bm x}^{(i)}\}\subset\mathbb{R}^n$ is called an exceptional family of elements for the pair $(F,G)$ with respect to the cone $K$, if the following conditions are satisfied:
\begin{itemize}
\itemindent 4pt
\item[{\rm (i)}] $\|{\bm x}^{(i)}\|\rightarrow\infty$ as $i\rightarrow\infty$;
\item[{\rm (ii)}] for every $i$, there exists a real number $\mu_i>0$ such that ${\bm s}^{(i)}=\mu_i{\bm x}^{(i)}+F({\bm x}^{(i)})\in K$, ${\bm v}^{(i)}=\mu_i{\bm x}^{(i)}+G({\bm x}^{(i)})\in K^*$, and $\langle{\bm s}^{(i)}, {\bm v}^{(i)}\rangle=0$.
\end{itemize}
\end{definition}

Using the notion mentioned above, we have the following result (see the proof in \cite{KI02}). Here, for the sake of completeness, we prove it again in a similar way used for proving \cite[Theorem 1]{IC99}.

\begin{lemma}\label{lemma3}
For two given continuous mappings $F,G:\mathbb{R}^n\rightarrow \mathbb{R}^n$ and a closed convex cone $K$ in $\mathbb{R}^n$, there exists either a solution of the problem ${\rm GCP}(F,G,K)$ or an exceptional family of elements for the pair $(F,G)$. \end{lemma}

\begin{proof}
Consider the function defined by
$$
\Phi(\bm x):=F(\bm x)-P_K(F(\bm x)-G(\bm x))
$$
for all $\bm x\in\mathbb{R}^n$. It is easy to show that the problem ${\rm GCP}(F,G,K)$ has a solution if and only if the equation $\Phi(\bm x)=\bm 0$ is solvable. For any given $\epsilon>0$, denote
$$
\mathbb{S}_\epsilon:=\{\bm x\in\mathbb{R}^n~|~\|\bm x\|=\epsilon\}\quad{\rm and}\quad\mathbb{B}_\epsilon:=\{\bm x\in\mathbb{R}^n~|~\|\bm x\|<\epsilon\}.
$$
 Consider the homotopy defined by
\begin{equation}\label{norm}
H(\bm x,t):=t{\bm x}+(1-t)\Phi(\bm x),~~~~0\leq t\leq1.
\end{equation}
From the definition of $\Phi$, we have
\begin{equation}\label{bnorm}
H(\bm x,t)=t{\bm x}+(1-t)F(\bm x)-(1-t)P_K(F(\bm x)-G(\bm x)), ~~0\leq t\leq1.
\end{equation}
We apply the topological degree theory and the Poincar\'{e}-B\"{o}hl Theorem in \cite{L78} for $\bm y=\bm 0$
and $\mathbb{B}_\epsilon ~(\partial\mathbb{B}_\epsilon=\mathbb{S}_\epsilon)$. We have the following two situations:

{\bf Case} (i). There exists $\epsilon>0$ such that $H(\bm x,t)\neq\bm 0$ for any $\bm x\in\mathbb{S}_\epsilon$ and $t\in[0,1]$. In this case, by Poincar\'{e}-B\"{o}hl Theorem, we have that ${\rm deg}(\Phi,\mathbb{B}_\epsilon,\bm 0)={\rm deg}(I,\mathbb{B}_\epsilon,\bm 0)$. Since ${\rm deg}(I,\mathbb{B}_\epsilon,\bm 0)=1$, we obtain ${\rm deg}(\Phi,\mathbb{B}_\epsilon,\bm 0)=1$, which means that equation $\Phi(\bm x)=\bm 0$ has a solution in $\mathbb{B}_\epsilon$, and hence the problem ${\rm GCP}(F,G,K)$ has a solution.

{\bf Case} (ii). For any $\epsilon>0$ there exist $\bm x^{(\epsilon)}\in\mathbb{S}_\epsilon$ and $t_\epsilon\in[0,1]$ such that
\begin{equation}\label{cnorm}
H(\bm x^{(\epsilon)},t_\epsilon)=\bm 0.
\end{equation}
We first claim that $t_\epsilon$ must be different from $1$. Indeed, if $t_\epsilon=1$ for some $\epsilon$, then by (\ref{norm}) and (\ref{cnorm}), we deduce that $\bm x^{(\epsilon)}=\bm 0$, which is impossible since $\bm x^{(\epsilon)}\in\mathbb{S}_\epsilon$. Secondly, if $t_\epsilon=0$ for some $\epsilon$, then (\ref{cnorm}) becomes $\Phi(\bm x^{(\epsilon)})=\bm 0$, which means that the problem ${\rm GCP}(F,G,K)$ has a solution.

Hence, we can say that either the problem ${\rm GCP}(F,G,K)$ has a solution, or for any $\epsilon>0$ there exist $\bm x^{(\epsilon)}\in\mathbb{S}_\epsilon$
and $t_\epsilon\in(0,1)$ such that $H(\bm x^{(\epsilon)},t_\epsilon)=\bm 0$. If ${\rm GCP}(F,G,K)$ has no solution, then from (\ref{bnorm}), we have
\begin{equation}\label{cnorm1}
\frac{t_\epsilon}{1-t_\epsilon}{\bm x^{(\epsilon)}}+F(\bm x^{(\epsilon)})=P_K(F(\bm x^{(\epsilon)})-G(\bm x^{(\epsilon)})).
\end{equation}
Consequently, by the basic properties of $P_K$ mentioned at the beginning of this section, it holds that
\begin{equation}\label{dnorm}
\left\langle\frac{t_\epsilon}{1-t_\epsilon}{\bm x^{(\epsilon)}}+F(\bm x^{(\epsilon)})-(F(\bm x^{(\epsilon)})-G(\bm x^{(\epsilon)})), {\bm y}\right\rangle\geq 0, {\rm~~~~~for~all~} {\bm y}\in K
\end{equation}
and
\begin{equation}\label{enorm}
\left\langle\frac{t_\epsilon}{1-t_\epsilon}{\bm x^{(\epsilon)}}+F(\bm x^{(\epsilon)})-(F(\bm x^{(\epsilon)})-G(\bm x^{(\epsilon)})), \frac{t_\epsilon}{1-t_\epsilon}{\bm x^{(\epsilon)}}+F(\bm x^{(\epsilon)})\right\rangle=0.
\end{equation}
By putting $\mu_\epsilon=t_\epsilon/(1-t_\epsilon)$ in (\ref{dnorm}) and (\ref{enorm}), we deduce
\begin{equation}\label{fnorm}
\langle\mu_\epsilon{\bm x^{(\epsilon)}}+G(\bm x^{(\epsilon)}), {\bm y}\rangle\geq 0, {\rm~~~~~for~all~} {\bm y}\in K
\end{equation}
and
\begin{equation}\label{gnorm}
\langle\mu_\epsilon{\bm x^{(\epsilon)}}+G(\bm x^{(\epsilon)}), \mu_\epsilon{\bm x^{(\epsilon)}}+F(\bm x^{(\epsilon)})\rangle=0.
\end{equation}
From (\ref{fnorm}), we know that $\mu_\epsilon{\bm x^{(\epsilon)}}+G(\bm x^{(\epsilon)})\in K^*$, which, together with (\ref{cnorm1}), (\ref{gnorm}) and the fact that $\|{\bm x}^{(\epsilon)}\|=\epsilon$ for any $\epsilon>0$, implies that $\{{\bm x}^{(\epsilon)}\}$ is an exceptional family of elements for the pair $(F,G)$.
\end{proof}

\begin{remark}\label{Rem1}
For Lemma \ref{lemma3}, if $S(K)\subset K$ holds, where $S(\bm x):=\bm x-F(\bm x)$ for any $\bm x\in \mathbb{R}^n$, then we know that the problem ${\rm GCP}(F,G,K)$ has either a solution in $K$, or an exceptional family of elements $\{\bm x^{(i)}\}$ in $K$ for $(F,G)$, see \cite{KI02}.
\end{remark}

Motivated by the concept of {\it exceptionally regular tensor} ({\bf ER}-tensor for short) introduced in \cite{WHB16}, we now define a new class of structured tensor pair that will play a key role in analyzing properties of ${\rm GPCP}(\Lambda,\bm a,\Theta,\bm b,K)$ .
\begin{definition}\label{ERdef}
 Let $K$ be a closed convex cone in $\mathbb{R}^n$ and $({\mathcal A},{\mathcal B})\in {\mathbb T}_{m,n}\times{\mathbb T}_{l,n}$. We say that $({\mathcal A},{\mathcal B})$ is
\begin{itemize}
\item[(i)] an ${\bf ER}^K$-tensor pair, if there exists no $({\bm x},v,t)\in(\mathbb{R}^n\backslash\{\bm
 0\})\times\mathbb{R}_+\times\mathbb{R}_+$ such that
\begin{equation}\label{equation1.2}
\left\{
\begin{array}{ll}
{\mathcal A}{\bm x}^{m-1}+v{\bm x}\in K,\\
{\mathcal B}{\bm x}^{l-1}+t{\bm x}\in K^*,\\
\left\langle {{\mathcal A}{\bm x}^{m-1}+v{\bm x},\mathcal B}{\bm x}^{l-1}+t{\bm x}\right\rangle=0.
\end{array}
\right.
\end{equation}
\item[(ii)] an ${\bf R}^{K}_0$-tensor pair, if there exists no ${\bm x}\in{\mathbb R}^n\backslash\{\bm 0\}$ such that
\begin{equation}\label{equation1.3}
{\mathcal A}{\bm x}^{m-1}\in K,\quad
{\mathcal B}{\bm x}^{l-1}\in K^*,\quad{\rm and}\quad
\left\langle{\mathcal A}{\bm x}^{m-1},{\mathcal B}{\bm x}^{l-1}\right\rangle=0.
\end{equation}
\end{itemize}
In particular, when $K:=\mathbb{R}_+^n$, we say simply that ${\bf ER}^K$-tensor pair and ${\bf R}^{K}_0$-tensor pair are ${\bf ER}$-tensor pair and ${\bf R}_0$-tensor pair, respectively.
\end{definition}

From the definitions above, we can see that an ${\bf ER}^K$-tensor pair must be an ${\bf R}^{K}_0$-tensor pair. If $K=\mathbb{R}_+^n$ and  $\mathcal{B}$ is an $n\times n$ identity matrix, then the above concept of ${\bf ER}$-tensor pair reduces to the concept of ${\bf ER}$-tensor introduced in \cite{WHB16}. Furthermore, it is easy to see that, if $(\mathcal{A},\mathcal{I})\in \mathbb{T}_{m,n}\times \mathbb{T}_{l,n}$ is an ${\bf ER}$-tensor pair, then $\mathcal{A}$ is an ${\bf ER}$-tensor; conversely, if $\mathcal{A}$ is an ${\bf ER}$-tensor and $l$ is even, then $(\mathcal{A},\mathcal{I})$ must be an ${\bf ER}$-tensor pair. In \cite{WHB16}, it has been proved that the class of strictly semi-positive tensors \cite{SQ14} is a subset of the class of ${\bf ER}$-tensors, and the class of weak {\bf P}-tensors \cite{WHB16} is also a proper subset of the class of ${\bf ER}$-tensors.

\begin{exam}\label{exam2-1}
Let $K=\mathbb{R}_+^2$ and ${\mathcal A}=(a_{i_1i_2i_3i_4})\in\mathbb{T}_{4,2}$ with $a_{1111}=1$, $a_{2111}=-1$, $a_{2222}=1$ and elsewhere zeros. Let ${\mathcal B}=(b_{i_1i_2i_3i_4})\in\mathbb{T}_{4,2}$ with $b_{1111}=1$, $b_{2122}=-1$, $b_{2222}=1$ and elsewhere zeros.
In this case, \eqref{equation1.2} is specified as
\begin{equation}\label{IC1}
\bm 0\leq\left(
\begin{array}{c}
x_1^3+vx_1\\
x_2^3-x_1^3+vx_2
\end{array}
\right)\perp\left(
\begin{array}{c}
x_1^3+tx_1\\
x_2^3-x_1x_2^2+tx_2
\end{array}
\right)\geq\bm 0,
\end{equation}
where ${\bm x}\in\mathbb{R}^n$ and $v,t\in\mathbb{R}_+$. For any $v,t\in\mathbb{R}_+$, from the first component in \eqref{IC1}, we have $x_1^2(x_1^2+v)(x_1^2+t)=0$, which implies that $x_1=0$. Consequently, from the second component in \eqref{IC1}, we have $x_2^2(x_2^2+v)(x_2^2+t)=0$, which implies that $x_2=0$. From the discussion above, we know that there exists no $({\bm x},v,t)\in(\mathbb{R}^2\backslash\{\bm
 0\})\times\mathbb{R}_+\times\mathbb{R}_+$ such that (\ref{IC1}) holds. Hence, $({\mathcal A},{\mathcal B})$ is an ${\bf ER}$-tensor pair.
\end{exam}

\section{Nonemptiness and compactness of the solution set}\label{SecNCS}

In this section, we mainly study the nonemptiness and compactness of the solution set of ${\rm GPCP}(\Lambda,\bm a,\Theta,\bm b,K)$ with structured tensors. Throughout, we denote by ${\rm SOL}(\Lambda,{\bm a},\Theta,{\bm b},K)$ the solution set of ${\rm GPCP}(\Lambda,\bm a,\Theta,\bm b,K)$ for notational convenience.

\begin{theorem}\label{theorem1.6}
 Let $\Lambda=\left({\mathcal A}^{(1)},\ldots,{\mathcal A}^{(m-1)}\right)\in \mathcal{F}_{m,n}$, $\Theta=\left({\mathcal B}^{(1)},\ldots,{\mathcal B}^{(l-1)}\right)\in \mathcal{F}_{l,n}$, and $K\subset\mathbb{R}^n$ a closed convex cone. Suppose that $({\mathcal A}^{(1)},{\mathcal B}^{(1)})$ is an ${\bf ER}^K$-tensor pair. If either (i) $m=l$ or (ii) $m\neq l$, but the one of $\mathcal{A}^{(1)}$ and $\mathcal{B}^{(1)}$, whose order is even and is larger than another, is positive definite (see \cite{Qi2005}), then for any given vectors $\bm a, \bm b\in \mathbb{R}^n$, the set ${\rm SOL}(\Lambda,\bm a,\Theta,\bm b, K)$ is nonempty and compact.
\end{theorem}

\begin{proof}
We first prove the nonemptiness of the solution set ${\rm SOL}(\Lambda,\bm a,\Theta,\bm b,K)$. Suppose, on the contrary, that ${\rm SOL}(\Lambda,\bm a,\Theta,\bm b,K)=\emptyset$. Then, it follows from Lemma \ref{lemma3} that there exists an exceptional family of elements for the pair $(F,G)$, i.e., there exists a sequence $\{{\bm x}^{(i)}\}_{i=1}^\infty\subset\mathbb{R}^n$ satisfying $\|{\bm x}^{(i)}\|\rightarrow\infty$ as $i\rightarrow\infty$ and, there exists a scalar $\mu_i>0$ for each $i$ such that
\begin{equation}\label{equation1.4}
\left\{
\begin{array}{lll}
\displaystyle\sum_{k=1}^{m-1}{\mathcal A}^{(k)}({\bm x}^{(i)})^{m-k}+\bm a+\mu_i{\bm x}^{(i)}\in K,\\
\displaystyle\sum_{p=1}^{l-1}{\mathcal B}^{(p)}({\bm x}^{(i)})^{l-p}+{\bm b}+\mu_i{\bm x}^{(i)}\in K^*,\\
\displaystyle h_i:=\left\langle\sum_{k=1}^{m-1}{\mathcal A}^{(k)}({\bm x}^{(i)})^{m-k}+\bm a+\mu_i{\bm x}^{(i)},\sum_{p=1}^{l-1}{\mathcal B}^{(p)}({\bm x}^{(i)})^{l-p}+{\bm b}+\mu_i{\bm x}^{(i)}\right\rangle=0.
\end{array}
\right.
\end{equation}
From the third expression in (\ref {equation1.4}), it is clear that ${h_i}/{\|\bm x^{(i)}\|^{m+l-2}}=0$ for every $i$, which shows
\begin{align}\label{Equat1}
0=&\frac{h_i}{\|{\bm x}^{(i)}\|^{m+l-2}}\nn \\
=&\frac{\left\langle\sum_{k=1}^{m-1}{\mathcal A}^{(k)}({\bm x}^{(i)})^{m-k}+\bm a,\sum_{p=1}^{l-1}{\mathcal B}^{(p)}({\bm x}^{(i)})^{l-p}+{\bm b}\right\rangle}{\|{\bm x}^{(i)}\|^{m+l-2}}\nn \\
&+\frac{t_i\left\langle{\bm x}^{(i)},\sum_{k=1}^{m-1}{\mathcal A}^{(k)}({\bm x}^{(i)})^{m-k}+\sum_{p=1}^{l-1}{\mathcal B}^{(p)}({\bm x}^{(i)})^{l-p}+{\bm a}+{\bm b}\right\rangle}{\|{\bm x}^{(i)}\|^m}+\frac{t_i^2}{\|{\bm x}^{(i)}\|^{m-l}},
\end{align}
where $t_i=\mu_i/\|{\bm x}^{(i)}\|^{l-2}$. We claim that $\{t_i\}_{i=1}^{\infty}$ is bounded. Suppose that $\{t_i\}_{i=1}^{\infty}$ is unbounded. Without loss of generality, we assume that $t_i\rightarrow +\infty$ as $i\rightarrow \infty$.

{\bf Case} (i). When $m=l$, by (\ref{Equat1}), it holds that
\begin{align}\label{Equat11}
0=&\frac{\left\langle\sum_{p=1}^{m-1}{\mathcal A}^{(p)}({\bm x}^{(i)})^{m-p}+\bm a,\sum_{p=1}^{m-1}{\mathcal B}^{(p)}({\bm x}^{(i)})^{m-p}+{\bm b}\right\rangle}{t_i^2\|{\bm x}^{(i)}\|^{2(m-1)}} \nn \\
&+\frac{\left\langle{\bm x}^{(i)}, \sum_{p=1}^{m-1}({\mathcal A}^{(p)}+{\mathcal B}^{(p)})({\bm x}^{(i)})^{m-p}+{\bm a}+{\bm b}\right\rangle}{t_i\|{\bm x}^{(i)}\|^m}+1.
\end{align}
Let $\bar{\bm x}^{(i)}={\bm x}^{(i)}/\|{\bm x}^{(i)}\|$. Since $\|\bar{\bm x}^{(i)}\|=1$ for any $i$, without loss of generality, we assume that $\bar{\bm x}^{(i)}\rightarrow\bar{\bm x}$ as $i\rightarrow \infty$. Consequently, it is easy to see that
\begin{align*}
\lim_{i\rightarrow \infty}\frac{\left\langle\sum_{p=1}^{m-1}{\mathcal A}^{(p)}({\bm x}^{(i)})^{m-p}+\bm a,\sum_{p=1}^{m-1}{\mathcal B}^{(p)}({\bm x}^{(i)})^{m-p}+{\bm b}\right\rangle}{\|{\bm x}^{(i)}\|^{2(m-1)}}=\left\langle {\mathcal A}^{(1)}\bar{\bm x}^{m-1},{\mathcal B}^{(1)}\bar{\bm x}^{m-1}\right\rangle
\end{align*}
and
$$\lim_{i\rightarrow \infty}\displaystyle\frac{\left\langle{\bm x}^{(i)}, \sum_{p=1}^{m-1}({\mathcal A}^{(p)}+{\mathcal B}^{(p)})({\bm x}^{(i)})^{m-p}+{\bm a}+{\bm b}\right\rangle}{\|{\bm x}^{(i)}\|^m}=({\mathcal A}^{(1)}+\mathcal{B}^{(1)})\bar{\bm x}^m,$$
which, together with (\ref{Equat11}), implies that $0=1$. It is a contradiction.

{\bf Case} (ii). When $m\neq l$, without loss of generality, we assume that $m>l$. Since ${\mathcal A}^{(1)}$ is positive definite and $\bar {\bm x}\in \mathbb{R}^n\backslash \{\bm 0\}$, it holds that
$$\lim_{i\rightarrow\infty}\frac{\left\langle{\bm x}^{(i)},\sum_{k=1}^{m-1}{\mathcal A}^{(k)}({\bm x}^{(i)})^{m-k}+\sum_{p=1}^{l-1}{\mathcal B}^{(p)}({\bm x}^{(i)})^{l-p}+{\bm a}+{\bm b}\right\rangle}{\|{\bm x}^{(i)}\|^m}={\mathcal A}^{(1)}\bar{\bm x}^m>0.$$
Consequently, from (\ref{Equat1}),  we obtain
\begin{align}\label{Equat2}
0\geq &\frac{\left\langle \sum_{k=1}^{m-1}{\mathcal A}^{(k)}({\bm x}^{(i)})^{m-k}+\bm a,\sum_{p=1}^{l-1}{\mathcal B}^{(p)}({\bm x}^{(i)})^{l-p}+{\bm b}\right\rangle}{\|{\bm x}^{(i)}\|^{m+l-2}} \nn\\
&+\frac{t_i\left\langle{\bm x}^{(i)},\sum_{k=1}^{m-1}{\mathcal A}^{(k)}({\bm x}^{(i)})^{m-k}+\sum_{p=1}^{l-1}{\mathcal B}^{(p)}({\bm x}^{(i)})^{l-p}+{\bm a}+{\bm b}\right\rangle}{\|{\bm x}^{(i)}\|^m}\nn \\
&\rightarrow+\infty,
\end{align}
which is also a contradiction.

Since $\{t_i\}_{i=1}^{\infty}$ is bounded, without loss of generality, we assume $\lim_{i\rightarrow\infty}t_i=\bar t$. It is clear that $\bar t\geq 0$. By the first two expressions in (\ref{equation1.4}), we know that, if $m=l$, then $${\mathcal A}^{(1)}\bar{\bm x}^{m-1}+\bar{t}\bar{\bm x}\in K,~{\mathcal B}^{(1)}\bar{\bm x}^{l-1}+\bar{t}\bar{\bm x}\in K^*,~\left\langle {\mathcal A}^{(1)}{\bar{\bm x}}^{m-1}+\bar{t}\bar{\bm x},{\mathcal B}^{(1)}\bar{\bm x}^{l-1}+\bar{t}\bar{\bm x}\right\rangle=0,$$
and if $m\neq l$, then
$$
\left\{
\begin{array}{ll}
{\mathcal A}^{(1)}\bar{\bm x}^{m-1}\in K,~{\mathcal B}^{(1)}\bar{\bm x}^{l-1}+\bar{t}\bar{\bm x}\in K^*,~\left\langle {\mathcal A}^{(1)}{\bar{\bm x}}^{m-1},{\mathcal B}^{(1)}\bar{\bm x}^{l-1}+\bar{t}\bar{\bm x}\right\rangle=0,&{\rm if~} m>l;\\
{\mathcal A}^{(1)}\bar{\bm x}^{m-1}+\bar t\bar {\bm x}\in K,~{\mathcal B}^{(1)}\bar{\bm x}^{l-1}\in K^*,~\left\langle {\mathcal A}^{(1)}{\bar{\bm x}}^{m-1}+\bar{t}\bar{\bm x},{\mathcal B}^{(1)}\bar{\bm x}^{l-1}\right\rangle=0,&{\rm otherwise}.
\end{array}\right.$$

From the discussion above, we know that there exists $(\hat{\bm x},\hat{v},\hat{t})\in(\mathbb{R}^n\backslash\{\bm
0\})\times\mathbb{R}_+\times\mathbb{R}_+$ satisfying the system (\ref{equation1.2}), which contradicts the condition that $({\mathcal A}^{(1)},{\mathcal B}^{(1)})$ is an ${\bf ER}^K$-tensor pair. Thus, ${\rm SOL}(\Lambda,\bm a,\Theta, \bm b,K)$ is nonempty.

Secondly, we prove that ${\rm SOL}(\Lambda,\bm a,\Theta,\bm b,K)$ is compact. Since $K$ and $K^*$ are closed, it is easy to see that ${\rm SOL}(\Lambda,\bm a,\Theta,\bm b,K)$ is closed. We now prove the boundedness of ${\rm SOL}(\Lambda,\bm a,\Theta,\bm b,K)$. Suppose that ${\rm SOL}(\Lambda,\bm a,\Theta,\bm b,K)$ is unbounded. Then there exists a sequence $\{{\bm x}^{(i)}\}^{\infty}_{i=1}\subseteq {\rm SOL}(\Lambda,\bm a,\Theta,\bm b,K)$ such that $\|{\bm x}^{(i)}\|\rightarrow\infty$ as $i\rightarrow\infty$. Without loss of generality, we assume that ${\bm x}^{(i)}/\|{\bm x}^{(i)}\|\rightarrow\bar{\bm x}$ as $i\rightarrow\infty$. It is clear that $\bar {\bm x}\neq \bm 0$. Furthermore, it follows from the definition of ${\rm GPCP}(\Lambda,\bm a,\Theta,\bm b,K)$ that
$$
\left\{
\begin{array}{l}
\displaystyle\frac{\bm a}{\|{\bm x}^{(i)}\|^{m-1}}+\sum_{k=1}^{m-1}\frac{1}{\|{\bm x}^{(i)}\|^{k-1}}{\mathcal A}^{(k)}({\bm x}^{(i)}/\|{\bm x}^{(i)}\|)^{m-k}\in K,\\
\displaystyle\frac{\bm b}{\|{\bm x}^{(i)}\|^{l-1}}+\sum_{p=1}^{l-1}\frac{1}{\|{\bm x}^{(i)}\|^{p-1}}{\mathcal B}^{(p)}({\bm x}^{(i)}/\|{\bm x}^{(i)}\|)^{l-p}\in K^*
\end{array}
\right.
$$
and
$$\displaystyle\frac{\left\langle\sum_{k=1}^{m-1}{\mathcal A}^{(k)}({\bm x}^{(i)})^{m-k}+\bm a,\sum_{p=1}^{l-1}{\mathcal B}^{(p)}({\bm x}^{(i)})^{l-p}+{\bm b}\right\rangle}{\|{\bm x}^{(i)}\|^{m+l-2}}=\frac{0}{\|{\bm x}^{(i)}\|^{m+l-2}}=0.
$$
Consequently, by letting $i\rightarrow\infty$, we have
\begin{equation}\label{BBYT}
{\mathcal A}^{(1)}\bar{\bm x}^{m-1}\in K,~~{\mathcal B}^{(1)}\bar{\bm x}^{l-1}\in K^*,~~{\rm and}~~\left\langle{\mathcal A}^{(1)}\bar{\bm x}^{m-1},{\mathcal B}^{(1)}\bar{\bm x}^{l-1}\right\rangle=0.
\end{equation}
This means that $\bar{\bm x}\in\mathbb{R}^n\backslash\{\bm 0\}$ satisfies the system (\ref{equation1.3}), which further contradicts the fact that $({\mathcal A}^{(1)},{\mathcal B}^{(1)})$ is an ${\bf R}^{K}_0$-tensor pair, since $({\mathcal A}^{(1)},{\mathcal B}^{(1)})$ is an ${\bf ER}^K$-tensor pair.  Thus ${\rm SOL}(\Lambda,\bm a,\Theta,\bm b,K)$ is compact.
\end{proof}

 \begin{remark}\label{Rem2}
 When $m>l$, if $S(K)\subset K$ where $S(\bm x):=\bm x-F(\bm x)$ for any $\bm x\in \mathbb{R}^n$, then by Remark \ref{Rem1} and a similar way to that used in Theorem \ref{theorem1.6}, we can obtain that ${\rm SOL}(\Lambda,\bm a,\Theta,\bm b,K)$ is nonempty and compact, provided $\mathcal{A}^{(1)}$ is strictly $K$-positive, i.e., $\mathcal{A}^{(1)}\bm x^{m}>0$ for any $\bm x\in K\backslash\{\bm 0\}$. Notice that, in this case, $m$ is not necessarily even.
  \end{remark}

\section{Some basic topological properties of the solution set}\label{Topprop}
 In this section, we further study some topological properties of ${\rm SOL}(\Lambda,\bm a,\Theta,\bm b,K)$. It is obvious that $\mathcal{F}_{m,n}$ is a linear space for any given positive integers $m$ and $n$. The distance between two elements $\Lambda_i=(\mathcal{A}^{(1)}_i,\mathcal{A}^{(2)}_i,\cdots,\mathcal{A}^{(m-1)}_i)\in\mathcal{F}_{m,n}~(i=1,2)$ is measured by means of
the expression
$$
\|\Lambda_1-\Lambda_2\|_F=\sqrt{\sum_{k=1}^{m-1}\left\|\mathcal{A}_1^{(k)}-\mathcal{A}_2^{(k)}\right\|_F^2}.
$$

Denote by $\C(\mathbb{R}^n)$ the set of nonzero closed convex cones in $\mathbb{R}^n$, which is associated with the natural metric defined by
$$
\delta(K_1,K_2) := \sup_{\|\bm z\|\leq 1}|{\rm dist}({\bm z},K_1)- {\rm dist}({\bm z},K_2)|,
$$
where ${\rm dist}({\bm z},K) := {\rm inf}_{{\bm u}\in K}\|{\bm z}-{\bm u}\|$ stands for the distance from $\bm z$ to $K$. An equivalent
way of defining $\delta$ is
$$\delta(K_1,K_2) = {\rm haus}(K_1\cap B_n,K_2 \cap B_n),$$
where $B_n$ is the closed unit ball in $\mathbb{R}^n$, and
$$
{\rm haus}(C_1,C_2) := \max\left\{\sup_{{\bm z}\in C_1}{\rm dist}({\bm z},C_2), \sup_{{\bm z}\in C_2}{\rm dist}({\bm z},C_1)\right\}
$$
stands for the Hausdorff distance between the compact sets $C_1,C_2\subset \mathbb{R}^n$ (see \cite[pp. 85-86]{Be93}). General information on the metric $\delta$ can be consulted in
the book by Rockafellar and Wets \cite{RW98}. According to \cite{WW67}, the
operation $\delta: K\to  K^*$ is an isometry on the space $(\C(\mathbb{R}^n), \delta)$, that is to say,
$$
\delta(K_1^*,K_2^*)=\delta(K_1,K_2), ~~~{\rm for~all~} K_1,K_2\in \C(\mathbb{R}^n).
$$
The basic topological properties of the mapping ${\rm SOL}(\cdot) : \mathcal{F}_{m,n}\times\mathbb{R}^n\times \mathcal{F}_{l,n}\times\mathbb{R}^n\times\C(\mathbb{R}^n)\rightarrow 2^{\mathbb{R}}$ are listed
in the following propositions. These propositions are some GPCP versions of the results presented in \cite{LHQ16}. As far as semicontinuity concepts are concerned, we use the
following terminology (see \cite[Section 6.2]{Be93}).

\begin{definition}
Let $W$ and $Y$ be two topological spaces and $\bar w\in W$. The mapping $\Psi:W\rightarrow 2^Y$ is said to be upper semicontinuous at $\bar w$, if for every open set $U$ with $\Psi(\bar w)\subset U$, there exists an open neighborhood $V$ of $\bar w$ such that $\Psi(w)\subset U$ for each $w\in V$.
\end{definition}

\begin{proposition}\label{Propt5}
The following two statements are true:
\begin{itemize}
\itemindent4pt
\item[{\rm(i)}] The set $\Sigma:=\{(\Lambda,\bm a,\Theta,\bm b,K,\bm x) \in \mathcal{F}_{m,n}\times\mathbb{R}^n\times \mathcal{F}_{l,n}\times\mathbb{R}^n\times\C(\mathbb{R}^n)\times\mathbb{R}^n~|~\bm x\in {\rm SOL}(\Lambda,\bm a,\Theta,\bm b,K)\}$ is closed in the product space $\mathcal{F}_{m,n}\times\mathbb{R}^n\times \mathcal{F}_{l,n}\times\mathbb{R}^n\times\C(\mathbb{R}^n)$. In particular, for any
    $\bar \Xi:=(\bar \Lambda,\bar{\bm a},\bar \Theta,\bar{\bm b},\bar{K})\in\mathcal{F}_{m,n}\times\mathbb{R}^n\times \mathcal{F}_{l,n}\times\mathbb{R}^n\times\C(\mathbb{R}^n)$, ${\rm SOL}(\bar \Xi)$ is a closed subset of $\mathbb{R}^n$;
\item[{\rm(ii)}] Let $\bar \Xi:=(\bar \Lambda,\bar{\bm a},\bar \Theta,\bar{\bm b},\bar{K})\in \mathcal{F}_{m,n}\times\mathbb{R}^n\times \mathcal{F}_{l,n}\times\mathbb{R}^n\times\C(\mathbb{R}^n)$. If the leading tensor pair $(\bar{\mathcal{A}}^{(1)},\bar{\mathcal{B}}^{(1)})$ in $(\bar\Lambda,\bar \Theta)$ is an ${\bf R}^{\bar{K}}_0$-tensor pair, then the mapping ${\rm SOL}(\cdot)$ is locally bounded at $\bar \Xi$, i.e., $$\bigcup_{(\Lambda,\bm a,\Theta,\bm b,K)\in \mathcal{N}}{\rm SOL}(\Lambda,\bm a,\Theta,\bm b,K)$$ is bounded for some neighborhood $\mathcal{N}$ of $\bar \Xi$.
\end{itemize}
\end{proposition}

\begin{proof}
{\bf Item} (i). The closedness of $\Sigma$ amounts to saying that
\begin{equation*}\label{DPk}
\left.
\begin{array}{r}
\Xi_i:=(\Lambda_i,{\bm a}^{(i)},\Theta_i,{\bm b}^{(i)},K_i)\rightarrow \bar\Xi:=(\bar{\Lambda},\bar{\bm a},\bar{\Theta},\bar{\bm b},\bar{K})\\
{\bm x}^{(i)}\rightarrow \bar{\bm x}\\
{\bm x}^{(i)}\in {\rm SOL}(\Xi_i)
\end{array}
\right\}\Rightarrow \bar{\bm x}\in{\rm SOL}(\bar\Xi).
\end{equation*}
Since ${\bm x}^{(i)}\in {\rm SOL}(\Xi_i)$ for every $i$, we have
\begin{equation}\label{nuTSVCP}
K_i\ni\left(\sum_{k=1}^{m-1}{\mathcal A}^{(k)}_i({\bm x}^{(i)})^{m-k}+{\bm a}^{(i)}\right)\perp\left(\sum_{p=1}^{l-1}{\mathcal B}^{(p)}_i({\bm x}^{(i)})^{l-p}+{\bm b}^{(i)}\right)\in K_i^*.
\end{equation}
Consequently, since ${\bm x}^{(i)}\rightarrow \bar{\bm x}$ as $i\rightarrow\infty$ and $\delta(K_1^*,K_2^*)=\delta(K_1,K_2)$ for any $K_1,K_2\in \C(\mathbb{R}^n)$, by passing to the limit in (\ref{nuTSVCP}), one gets
$$
\bar K\ni\left\{\sum_{k=1}^{m-1}\bar{{\mathcal A}}^{(k)}\bar{\bm x}^{m-k}+\bar{\bm a}\right\}\perp\left\{\sum_{p=1}^{l-1}\bar{{\mathcal B}}^{(p)}\bar{{\bm x}}^{l-p}+\bar{{\bm b}}\right\}\in\bar K^*,
$$
which implies $\bar{\bm x}\in{\rm SOL}(\bar\Xi)$. We proved the first part (i) of this proposition.

{\bf Item} (ii). Suppose that the map ${\rm SOL}(\cdot)$ is not locally bounded at $\bar\Xi$. Then there exists sequences $\{\Xi_i=(\Lambda_i,{\bm a}^{(i)},\Theta_i,{\bm b}^{(i)},K_i)\}$ and $\{{\bm x}^{(i)}\}$ satisfying
$$
\|\Lambda_i-\bar{\Lambda}\|_F\rightarrow 0,~~\|{\bm a}^{(i)}-\bar{\bm a}\|\rightarrow 0,~~\|\Theta_i-\bar{\Theta}\|_F\rightarrow 0,~~\|{\bm b}^{(i)}-\bar{\bm b}\|\rightarrow 0,~~\delta(K_i,\bar K)\rightarrow 0,
$$
and $\|{\bm x}^{(i)}\|\rightarrow +\infty$ such that ${\bm x}^{(i)}\in {\rm SOL}(\Xi_i)$, i.e., (\ref{nuTSVCP}) holds for any $i=1,2,\ldots$. Without loss of generality, we assume that ${\bm x}^{(i)}/\|{\bm x}^{(i)}\|\rightarrow \bar{\bm x}$ as $i\rightarrow\infty$. It is obvious that $\|\bar{\bm x}\|=1$, which means $\bar {\bm x}\in \mathbb{R}^n\backslash\{\bm 0\}$. Furthermore, by (\ref{nuTSVCP}), we have
$$
\left\{
\begin{array}{l}
\displaystyle\frac{{\bm a}^{(i)}}{\|{\bm x}^{(i)}\|^{m-1}}+\sum_{k=1}^{m-1}\frac{1}{\|{\bm x}^{(i)}\|^{k-1}}{\mathcal A}^{(k)}_i({\bm x}^{(i)}/\|{\bm x}^{(i)}\|)^{m-k}\in K_i,\\
\displaystyle\frac{{\bm b}^{(i)}}{\|{\bm x}^{(i)}\|^{l-1}}+\sum_{p=1}^{l-1}\frac{1}{\|{\bm x}^{(i)}\|^{p-1}}{\mathcal B}^{(p)}_i({\bm x}^{(i)}/\|{\bm x}^{(i)}\|)^{l-p}\in K_i^*
\end{array}
\right.
$$
and
$$\displaystyle\frac{\left\langle\sum_{k=1}^{m-1}{\mathcal A}^{(k)}_i({\bm x}^{(i)})^{m-k}+{\bm a}^{(i)},\sum_{p=1}^{l-1}{\mathcal B}^{(p)}_i({\bm x}^{(i)})^{l-p}+{\bm b}^{(i)}\right\rangle}{\|{\bm x}^{(i)}\|^{m+l-2}}=0.
$$
By passing to the limit in the above expression, it holds that
$$
\bar{\mathcal A}^{(1)}\bar{\bm x}^{m-1}\in\bar{K},\quad
\bar{\mathcal B}^{(1)}\bar{\bm x}^{l-1}\in\bar{K}^*,\quad {\rm and} \quad
\left\langle\bar{\mathcal A}^{(1)}\bar{\bm x}^{m-1},\bar{\mathcal B}^{(1)}\bar{\bm x}^{l-1}\right\rangle=0.
$$
It contradicts the condition that $(\bar{\mathcal A}^{(1)},\bar{\mathcal B}^{(1)})$ is an ${\bf R}^{\bar{K}}_0$-tensor pair, because $\bar{\bm x}\in\mathbb{R}^n\backslash\{\bm 0\}$.
\end{proof}

\begin{proposition}\label{Propt7}
 Let $(\bar \Lambda,\bar \Theta,\bar{K})\in \mathcal{F}_{m,n}\times \mathcal{F}_{l,n}\times\C(\mathbb{R}^n)$. If the leading tensor pair $(\bar{\mathcal{A}}^{(1)},\bar{\mathcal{B}}^{(1)})$ in $(\bar\Lambda,\bar \Theta)$ is an ${\bf R}^{\bar{K}}_0$-tensor pair, then the set ${\rm D}_{\rm SOL}:=\{(\bm a,\bm b)\in \mathbb{R}^n\times\mathbb{R}^n~|~{\rm SOL}(\bar \Lambda,\bm a,\bar \Theta,\bm b,\bar{K})\neq\emptyset\}$ is closed.
\end{proposition}

\begin{proof}
Take any sequence $\{({\bm a}^{(i)},{\bm b}^{(i)})\}\subset {\rm D}_{\rm SOL}$ with $({\bm a}^{(i)},{\bm b}^{(i)})\rightarrow (\bar {\bm a},\bar {\bm b})$ as $i\rightarrow\infty$. Then, we just need to prove $(\bar {\bm a},\bar {\bm b})\in {\rm D}_{\rm SOL}$. Since ${\rm SOL}(\bar\Lambda,{\bm a}^{(i)},\bar\Theta,{\bm b}^{(i)},\bar K)\neq\emptyset$, let us pick up ${\bm x}^{(i)}\in {\rm SOL}(\bar \Lambda,{\bm a}^{(i)},\bar \Theta,{\bm b}^{(i)},\bar K)$ for every $i$. It follows immediately from {\bf Item} (ii) of Proposition \ref{Propt5} that $\{{\bm x}^{(i)}\}$ is bounded. Without loss of generality, we assume that $\bm x^{(i)}\rightarrow\bar{\bm x}$ as $i\rightarrow\infty$. Consequently, by {\bf Item} (i) of Proposition \ref{Propt5}, we know that $\bar{\bm x}\in {\rm SOL}(\bar \Lambda,\bar{\bm a},\bar \Theta,\bar{\bm b},\bar{K})$, which implies ${\rm SOL}(\bar \Lambda,\bar{\bm a},\bar \Theta,\bar{\bm b},\bar{K})\neq\emptyset$. Therefore, we obtain $(\bar {\bm a},\bar {\bm b})\in {\rm D}_{\rm SOL}$ and complete the proof.
\end{proof}

From Theorem \ref{theorem1.6} and Proposition \ref{Propt5}, we immediately have the following corollary.

\begin{corollary}\label{compact} Let $(\bar\Lambda,\bar{\bm a},\bar\Theta,\bar{\bm b},\bar K) \in \mathcal{F}_{m,n}\times\mathbb{R}^n\times\mathcal{F}_{l,n}\times\mathbb{R}^n\times\C(\mathbb{R}^n)$. If the leading tensor pair $(\bar{\mathcal A}^{(1)},\bar{\mathcal B}^{(1)})$ in $(\bar\Lambda, \bar\Theta )$ is an ${\bf R}^{\bar{K}}_0$-tensor pair, then ${\rm SOL}(\bar\Lambda,\bar{\bm a},\bar\Theta,\bar{\bm b},\bar K)$ is compact.
\end{corollary}

\begin{proposition}\label{Propt6}
Let $\bar\Xi:=(\bar \Lambda,\bar{\bm a},\bar \Theta,\bar{\bm b},\bar{K})\in \mathcal{F}_{m,n}\times\mathbb{R}^n\times \mathcal{F}_{l,n}\times\mathbb{R}^n\times\C(\mathbb{R}^n)$. If the leading tensor pair $(\bar{\mathcal A}^{(1)},\bar{\mathcal B}^{(1)})$ in $(\bar\Lambda, \bar\Theta )$ is an ${\bf R}^{\bar{K}}_0$-tensor pair, then the mapping ${\rm SOL}(\cdot)$ is upper semicontinuous at $\bar \Xi$.
\end{proposition}

\begin{proof}
Suppose, on the contrary, that the mapping ${\rm SOL}(\cdot)$ is not upper semicontinuous at $\bar\Xi$. Then we could find an open set $\bar U\subset \mathbb{R}^n$ with ${\rm SOL}(\bar\Xi)\subset \bar U$ and a sequence $\{\Xi_i:=(\Lambda_i,{\bm a}^{(i)},\Theta_i,{\bm b}^{(i)},K_i)\}$ satisfying $\Xi_i\rightarrow \bar\Xi$, such that ${\rm SOL}(\Xi_i)\cap (\mathbb{R}^n\backslash \bar U)\neq\emptyset
$
for any $i=1,2,\cdots$. Now, for each $i$, pick up ${\bm x}^{(i)}\in{\rm SOL}(\Xi_i)\cap (\mathbb{R}^n\backslash \bar U)$. It follows from {\bf Item} (ii) of Proposition \ref{Propt5} that the sequence $\{{\bm x}^{(i)}\}$ admits a converging subsequence. By {\bf Item} (i) of Proposition \ref{Propt5}, the corresponding
limit must be in ${\rm SOL}(\bar\Xi)\cap (\mathbb{R}\backslash\bar U)$, which, together with  ${\rm SOL}(\bar\Xi)\subset \bar U$, leads to a contradiction.
\end{proof}

\section{Error bound analysis}\label{SecEA}
Among all the useful tools for theoretical and numerical treatment to NCPs, the global error bound, i.e., an upper bound estimation of the distance from a given point in $\mathbb{R}^n$ to the solution set of the
problem in terms of some residual functions, is an important one. In this section, we focus on studying the global error bound of the solution set of ${\rm GPCP}(\Lambda,\bm a,\Theta,\bm b,K)$ with $K:=\mathbb{R}^n_+$, and in particular, we denote the resulting model by ${\rm GPCP}(\Lambda,\bm a,\Theta,\bm b)$ for simplicity. Moreover,
we use $\Omega$ to represent the solution set of ${\rm GPCP}(\Lambda,\bm a,\Theta,\bm b)$. Let $r(\cdot):\mathbb{R}^n\rightarrow\mathbb{R}$. We say that $r(\bm x)$ is a residual function of ${\rm GPCP}(\Lambda,\bm a,\Theta,\bm b)$, if $r(\bm x)\geq0$, and $r(\bm x)=0$ if and only if $\bm x\in\Omega$. For the two polynomials $F,G$ given in \eqref{gpcp}, let $\min \{F(\bm x), G(\bm x)\}$ denote the vector in $\mathbb{R}^n$ whose $j$-th component is $\min \{F_j(\bm x), G_j(\bm x)\}$ for $j\in [n]$. Obviously, the function $r(\bm x)$, defined by
\begin{equation}\label{ICP52}
r(\bm x)=\|\min\{F(\bm x),G(\bm x)\}\|,
\end{equation}
provides a residual function of ${\rm GPCP}(\Lambda,\bm a,\Theta,\bm b)$, and it is called the natural residual for
${\rm GPCP}(\Lambda,\bm a,\Theta,\bm b)$ (see \cite{KF96}).

\begin{definition}\label{ICPD51}
Let $r(\bm x)$ be a residual function of ${\rm GPCP}(\Lambda,\bm a,\Theta,\bm b)$. We say $r(\bm x)$ is a local error bound for
${\rm GPCP}(\Lambda,\bm a,\Theta,\bm b)$, if there exist three constants $c>0$, $\tau\in(0,1]$ and $\varepsilon>0$ such that
\begin{equation}\label{errorbound}
{\rm dist}(\bm x,\Omega)\leq c~ r(\bm x)^\tau
\end{equation}
holds for any $\bm x\in\mathbb{R}^n$ with $r(\bm x)\leq\varepsilon$, where ${\rm dist}(\bm x,\Omega)$ denotes the distance between the point $\bm x$ and the set $\Omega$. Furthermore, if \eqref{errorbound} holds for any $\bm x\in\mathbb{R}^n$, then $r(\bm x)$ is called as a global error bound for ${\rm GPCP}(\Lambda,\bm a,\Theta,\bm b)$. If \eqref{errorbound} holds for $\tau=1$, $r(\bm x)$ is called as a Lipschitzian error bound.
\end{definition}

Below, we first list the fundamental assumption used in \cite{XZ02} for the analysis of the global error bound of ${\rm GCP}(F,G,\mathbb{R}_+^n)$.

\begin{assumption}\label{ICPA40}
\begin{itemize}
\item[{\rm(i)}] There exists a  constant  $\rho>0$  such  that
\begin{equation}\label{ICP50}
\max_{j\in[n]}~[F_j(\bm x)-F_j(\bm y)][G_j(\bm x)-G_j(\bm y)]>\rho\|\bm x-\bm y\|^2,~~~{\rm for~ all~} \bm x,\bm y\in\mathbb{R}^n.
\end{equation}
\item[{\rm(ii)}] There exists a  constant  $\mu>0$  such  that
\begin{equation}\label{ICP501}
[F(\bm x)-F(\bm y)]^\t [G(\bm x)-G(\bm y)]>\mu\|\bm x-\bm y\|^2,~~~{\rm for~ all~}\bm x,\bm y\in\mathbb{R}^n.
\end{equation}
\end{itemize}
\end{assumption}

It is easy to see that a pair of functions $(F,G)$ satisfying (\ref{ICP501}) must meet (\ref{ICP50}).
In \cite{KF96}, the authors proved that the natural
residual $r(\bm x)$ defined by \eqref{ICP52} provides a global error bound for ${\rm GCP}(F,G,\mathbb{R}_+^n)$ under Item (i) of Assumption \ref{ICPA40}. Unfortunately, Assumption \ref{ICPA40} often does not hold when we discuss ${\rm GCP}(F,G,\mathbb{R}_+^n)$, especially in the case where the related functions $F$ and $G$ are both polynomials with high degree. To overcome this inadequacy, we now introduce the following assumption, which is slightly weaker than \cite[Assumption 5.1]{Hua05}.

\begin{assumption}\label{ICPA51}
For any sequence $\{{\bm x}^{(i)}\}$ in $\mathbb{R}^n$ with ${\bm x}^{(i)}\rightarrow\bar{\bm x}\in\Omega$ as $i\rightarrow\infty$ and ${\bm x}^{(i)}\neq\bar{\bm x}$, there exist a subsequence $\{\bm x^{(i_j)}\}$ of $\{\bm x^{(i)}\}$ and an index $j_0\in[n]$ such that
\begin{equation}\label{ICP53}
\lim_{i_j\rightarrow\infty}\frac{{\rm min}\left\{F_{j_0}({\bm x}^{(i_j)}),G_{j_0}({\bm x}^{(i_j)})\right\}}{\|{\bm x}^{(i_j)}-\bar{\bm x}\|}\neq0.
\end{equation}
\end{assumption}

Below, we use an example to illustrate that satisfying Assumption \ref{ICPA51} does not necessarily mean satisfying Assumption \ref{ICPA40}.

\begin{exam}\label{exam5.1}
We still consider the tensor pair $(\mathcal{A},\mathcal{B})$ in Example \ref{exam2-1}. Let $F({\bm x})={\mathcal A}{\bm x}^3+\bm a$ and $G({\bm x})={\mathcal B}{\bm x}^3+\bm b$, with $\bm a=(-1,0)^\top$ and $\bm b=(1,0)^\top$.
We first claim that $(F(\bm x),G(\bm x))$ does not satisfy \eqref{ICP50}. Suppose that the pair $(F(\bm x),G(\bm x))$ satisfies \eqref{ICP50}, i.e., there exists a number $\rho>0$ such that \eqref{ICP50} holds. Let $\bm x=(\varepsilon,\varepsilon/2)^\top$ and $\bm y=(\varepsilon,\varepsilon)^\top$ with $\varepsilon>0$. Then, by \eqref{ICP50} and the fact that $\varepsilon>0$, we have $7\varepsilon^4\geq 16\rho$. Consequently, by letting $\varepsilon \rightarrow 0$, we obtain a contradiction, since $\rho >0$.

We now can claim that Assumption \ref{ICPA51} holds for $(F,G)$. First, it is easy to see that the solution set $\Omega=\{\bar {\bm x}=(1,1)^\top\}$. For any sequence $\{{\bm x}^{(i)}\}$ in $\mathbb{R}^2$ with $\bm x^{(i)}\rightarrow \bar{\bm x}$ as $i\rightarrow\infty$ and ${\bm x}^{(i)}\neq \bar {\bm x}$, let us write $\bm x^{(i)}=(1+\tau_i,1+\varepsilon_i)^\top$, where $(\tau_i,\varepsilon_i)^\top\rightarrow(0,0)^\top$ as $i\rightarrow\infty$. For the considered pair $(\mathcal{A},\mathcal{B})$, we have
\begin{equation*}\label{YYH1}
\min\{F_1(\bm x^{(i)}),G_1(\bm x^{(i)})\}=F_1(\bm x^{(i)})=\tau_i^3+3\tau_i^2+3\tau_i
\end{equation*}
and
\begin{equation*}\label{ICPJ1}
\min\{F_2(\bm x^{(i)}),G_2(\bm x^{(i)})\}=\left\{
\begin{array}{ll}
\varepsilon_i^3+3\varepsilon_i^2+3\varepsilon_i-\tau_i^3-3\tau_i^2-3\tau_i,&{\rm if}~\tau_i\geq\varepsilon_i,\\
(1+\varepsilon_i)^2(\varepsilon_i-\tau_i),&{\rm if}~\tau_i<\varepsilon_i.
\end{array}
\right.
\end{equation*}
It is obvious that, if $\tau_i=0$ for any $i$, then $\varepsilon_i\neq 0$ for any $i$, and
\begin{equation*}\label{2distbar0}
\lim_{i\rightarrow\infty}\frac{{\rm min}\left\{F_2({\bm x}^{(i)}),G_2({\bm x}^{(i)})\right\}}{\|{\bm x}^{(i)}-\bar{\bm x}\|}
=\left\{
\begin{array}{ll}
-\displaystyle\lim_{i\rightarrow\infty}\varepsilon_i^2+3\varepsilon_i+3=-3,&{\rm if}~0\geq\varepsilon_i,\\
\displaystyle\lim_{i\rightarrow\infty}(1+\varepsilon_i)^2=1,&{\rm if}~0<\varepsilon_i.\end{array}
\right.
\end{equation*}
Hence, \eqref{ICP53} holds for $i_0=2$. Similarly, if $\varepsilon_i=0$ for any $i$, then we can obtain that \eqref{ICP53} holds for $i_0=1$. Now we discuss the case where $\{\tau_i\}$ includes a subsequences $\{\tau_{i_j}\}$ with $\tau_{i_j}\neq 0$ for any $j$, and $\{\varepsilon_i\}$ includes a subsequences $\{\varepsilon_{i_s}\}$ with $\varepsilon_{i_s}\neq 0$ for any $s$. Without loss of generality, we assume that $\tau_i,\varepsilon_i\neq 0$ for any $i$. In this case, if $
\lim_{i\rightarrow\infty}\frac{{\rm min}\left\{F_1({\bm x}^{(i)}),G_1({\bm x}^{(i)})\right\}}{\|{\bm x}^{(i)}-\bar{\bm x}\|}$ does not exist, or exists but equals to nonzero, then \eqref{ICP53} holds for $i_0=1$.
Now we assume that $
\lim_{i\rightarrow\infty}\frac{{\rm min}\left\{F_1({\bm x}^{(i)}),G_1({\bm x}^{(i)})\right\}}{\|{\bm x}^{(i)}-\bar{\bm x}\|}=0$, which means
\begin{equation}\label{TTRE}
\lim_{i\rightarrow\infty}\frac{3+3\tau_i+\tau_i^2}{\sqrt{1+(\varepsilon_i/\tau_i)^2}}=0.
\end{equation}
Since $\tau_i\rightarrow 0$ as $i\rightarrow\infty$, by \eqref{TTRE} we know $\lim_{i\rightarrow\infty}\varepsilon_i/\tau_i=\infty$.
Thus, we have
\begin{equation*}\label{2distbar}
\lim_{i\rightarrow\infty}\frac{{\rm min}\left\{F_2({\bm x}^{(i)}),G_2({\bm x}^{(i)})\right\}}{\|{\bm x}^{(i)}-\bar{\bm x}\|}
=\left\{
\begin{array}{ll}
\begin{array}{l}
\displaystyle\lim_{i\rightarrow\infty}\frac{\varepsilon_i^3+3\varepsilon_i^2+3\varepsilon_i-\tau_i^3-3\tau_i^2-3\tau_i}{\sqrt{\tau_i^2+\varepsilon_i^2}}\\
\qquad\;\;=\pm\displaystyle\lim_{i\rightarrow\infty}\frac{\varepsilon_i^2+3\varepsilon_i+3}{\sqrt{(\tau_i/\varepsilon_i)^2+1}}=\pm 3,\\
\end{array}&{\rm if}~\tau_i\geq\varepsilon_i,\\
\pm\displaystyle\lim_{i\rightarrow\infty}\frac{(1+\varepsilon_i)^2(1-\tau_i/\varepsilon_i)}{\sqrt{(\tau_i/\varepsilon_i)^2+1}}=\pm1
,&{\rm if}~\tau_i<\varepsilon_i,\end{array}
\right.
\end{equation*}
Therefore, \eqref{ICP53} holds for $i_0=2$. From the discussion above, we know that $(F,G)$ satisfies Assumption \ref{ICPA51}.
\end{exam}

Now we further assume the following conditions.
\begin{condition}\label{ICPC53}
Given $\Lambda=\left({\mathcal A}^{(1)},\ldots,{\mathcal A}^{(m-1)}\right)\in \mathcal{F}_{m,n}$, $\Theta=\left({\mathcal B}^{(1)},\ldots,{\mathcal B}^{(l-1)}\right)\in \mathcal{F}_{l,n}$ and ${\bm a, \bm b}\in\mathbb{R}^n$. If there exists a sequence $\{{\bm x}^{(i)}\}$ satisfying $\|{\bm x}^{(i)}\|\rightarrow\infty$ such that
\begin{equation}\label{ICP54}
 \frac{\left[-\left(\sum_{k=1}^{m-1}{\mathcal A}^{(k)}({\bm x}^{(i)})^{m-k}+{\bm a}\right)\right]_+}{\|{\bm x}^{(i)}\|}\rightarrow\bm 0,~~\frac{\left[-\left(\sum_{p=1}^{l-1}{\mathcal B}^{(p)}({\bm x}^{(i)})^{l-p}+{\bm b}\right)\right]_+}{\|{\bm x}^{(i)}\|}\rightarrow\bm 0,
\end{equation}
as $i\rightarrow\infty$, then there exists an index $j_0\in[n]$ such that
\begin{equation}\label{ICP554}
\limsup_{i\rightarrow\infty}\frac{\min\left\{\left(\sum_{k=1}^{m-1}{\mathcal A}^{(k)}({\bm x}^{(i)})^{m-k}+{\bm a}\right)_{j_0},\left(\sum_{p=1}^{l-1}{\mathcal B}^{(p)}({\bm x}^{(i)})^{l-p}+{\bm b}\right)_{j_0}\right\}}{\|{\bm x}^{(i)}\|}> 0.
\end{equation}
\end{condition}

Then, we have the following lemma.

\begin{lemma}\label{ICPlemma}
Given $\Lambda=\left({\mathcal A}^{(1)},\ldots,{\mathcal A}^{(m-1)}\right)\in \mathcal{F}_{m,n}$, $\Theta=\left({\mathcal B}^{(1)},\ldots,{\mathcal B}^{(l-1)}\right)\in \mathcal{F}_{l,n}$ and ${\bm a, \bm b}\in\mathbb{R}^n$. If $({\mathcal A}^{(1)},{\mathcal B}^{(1)})$ is an ${\bf R}_0$-tensor pair, then Condition \ref{ICPC53} holds.
\end{lemma}
\begin{proof}
We prove that Condition \ref{ICPC53} holds by contradiction. Suppose there exists a sequence $\{{\bm x}^{(i)}\}\subset\mathbb{R}^n$ satisfying \eqref{ICP54}, such that (\ref {ICP554}) does not hold. Then for every given $\bar j\in [n]$, it holds that
\begin{equation}\label{ICP55}
\limsup_{i\rightarrow \infty}\frac{\min\left\{\left(\sum_{k=1}^{m-1}{\mathcal A}^{(k)}({\bm x}^{(i)})^{m-k}+{\bm a}\right)_{\bar j},\left(\sum_{p=1}^{l-1}{\mathcal B}^{(p)}({\bm x}^{(i)})^{l-p}+{\bm b}\right)_{\bar j}\right\}}{\|{\bm x}^{(i)}\|}\leq 0.
\end{equation}

Since the sequence $\left\{{\bm x}^{(i)}/\|{\bm x}^{(i)}\|\right\}$ is bounded, without loss of generality, we assume ${\bm x}^{(i)}/\|{\bm x}^{(i)}\|\rightarrow\bar{\bm x}$ as $i\rightarrow\infty$. It is obvious that $\bar{\bm x}\in\mathbb{R}^n\backslash\{\bm 0\}$. Hereafter, we claim that ${\mathcal A}^{(1)}\bar{\bm x}^{m-1}\geq \bm 0$ and ${\mathcal B}^{(1)}\bar{\bm x}^{l-1}\geq \bm 0$.
In fact, for every given $\bar j\in [n]$, if there exists a subsequence $\{\bm x^{(i_j)}\}$ of $\{{\bm x}^{(i)}\}$ so that 
$$\left(\sum_{k=1}^{m-1}{\mathcal A}^{(k)}({\bm x}^{(i_j)})^{m-k}+{\bm a}\right)_{\bar j}\leq0\;\;\;{\rm for~every}\;\; i_j.$$ Then, we have
\begin{align*}
0&=\lim_{i_j\rightarrow\infty}\frac{\left[-\left(\sum_{k=1}^{m-1}{\mathcal A}^{(k)}({\bm x}^{(i_j)})^{m-k}+{\bm a}\right)_{\bar j}\right]_+}{\|\bm
x^{(i_j)}\|}\\
&=\lim_{i_j\rightarrow\infty}\frac{-\left(\sum_{k=1}^{m-1}{\mathcal A}^{(k)}({\bm x}^{(i_j)})^{m-k}+{\bm a}\right)_{\bar j}}{\|{\bm x}^{(i_j)}\|^{m-1}}\\
&=-({\mathcal A}^{(1)}\bar{\bm x}^{m-1})_{\bar j};
\end{align*}
otherwise, we obtain
\begin{align*}
0&\leq\lim_{i_j\rightarrow\infty}\frac{\left(\sum_{k=1}^{m-1}{\mathcal A}^{(k)}({\bm x}^{(i_j)})^{m-k}+{\bm a}\right)_{\bar j}}{\|\bm
x^{(i_j)}\|^{m-1}}\\
&=\lim_{i_j\rightarrow\infty}\frac{\left({\mathcal A}^{(1)}({\bm x}^{(i_j)})^{m-1}\right)_{\bar j}}{\|{\bm x}^{(i_j)}\|^{m-1}}\\
&=({\mathcal A}^{(1)}\bar{\bm x}^{m-1})_{\bar j}.
\end{align*}
By combining the two situations above together, we obtain that $({\mathcal A}^{(1)}\bar{\bm x}^{m-1})_{\bar j}\geq0$ for every ${\bar j}\in [n]$. Similarly, we also can obtain that $({\mathcal B}^{(1)}\bar{\bm x}^{l-1})_{\bar j}\geq0$ for every ${\bar j}\in [n]$.

On the other hand, if there exists a subsequence $\{\bm x^{(i_j)}\}$ of $\{\bm x^{(i)}\}$ such that
\begin{equation}\label{TTHe1}
\left(\sum_{k=1}^{m-1}{\mathcal A}^{(k)}({\bm x}^{(i_j)})^{m-k}+{\bm a}\right)_{\bar j}\leq\left(\sum_{p=1}^{l-1}{\mathcal B}^{(p)}({\bm x}^{(i_j)})^{l-p}+{\bm b}\right)_{\bar j}, ~~~~\forall~ i_j,
 \end{equation}
 then by \eqref{ICP55}, we obtain
$$
\limsup_{i_j\rightarrow\infty}\frac{\left(\sum_{k=1}^{m-1}{\mathcal A}^{(k)}({\bm x}^{(i_j)})^{m-k}+{\bm a}\right)_{\bar j}}{\|\bm x^{(i_j)}\|}\leq 0.
$$
Consequently, we know that
$$
(\mathcal{A}^{(1)}\bar{\bm x}^{m-1})_{\bar j}=\lim_{i_j\rightarrow\infty}\frac{\left(\sum_{k=1}^{m-1}{\mathcal A}^{(k)}({\bm x}^{(i_j)})^{m-k}+{\bm a}\right)_{\bar j}}{\|\bm x^{(i_j)}\|^{m-1}}\leq 0,
$$
which, together with the obtained result that $(\mathcal{A}^{(1)}\bar{\bm x}^{m-1})_{\bar j}\geq 0$, implies that $(\mathcal{A}^{(1)}\bar{\bm x}^{m-1})_{\bar j}=0$. If (\ref{TTHe1}) does not hold, then in a similar way, we may know that
$(\mathcal{B}^{(1)}\bar{\bm x}^{l-1})_{\bar j}\leq 0$, and hence we have $(\mathcal{B}^{(1)}\bar{\bm x}^{l-1})_{\bar j}=0$.

From the discussion above, we know that $\bar{\bm x}\in\mathbb{R}^n\backslash\{\bm
0\}$ satisfies the system (\ref{equation1.3}) with $K=\mathbb{R}_+^n$, which contradicts the condition that $({\mathcal A}^{(1)},{\mathcal B}^{(1)})$ is an ${\bf R}_0$-tensor pair. Thus Condition \ref{ICPC53} holds.
\end{proof}

We now present the main result in this section, which shows that the natural residual function $r(\bm x)$ defined in (\ref{ICP52}) is a global Lipschitzian error bound for ${\rm GPCP}(\Lambda,\bm a,\Theta,\bm b)$ under some appropriate conditions.
\begin{theorem}\label{ICPI51}
Given $\Lambda=\left({\mathcal A}^{(1)},\ldots,{\mathcal A}^{(m-1)}\right)\in \mathcal{F}_{m,n}$, $\Theta=\left({\mathcal B}^{(1)},\ldots,{\mathcal B}^{(l-1)}\right)\in \mathcal{F}_{l,n}$ and ${\bm a, \bm b}\in\mathbb{R}^n$.  Suppose that Condition \ref{ICPC53} holds and $\Omega$ is nonempty. If Assumption \ref{ICPA51} holds, then $r(\bm x)$ is a global Lipschitzian error bound for ${\rm GPCP}(\Lambda,\bm a,\Theta,\bm b)$.
\end{theorem}

\begin{proof}
The proof is divided into two parts. Concretely, we first prove that $r(\bm x)$ is a local Lipschitzian error bound for ${\rm GPCP}(\Lambda,\bm a,\Theta,\bm b)$. Then, we prove that $r(\bm x)$ is a global Lipschitzian error bound for ${\rm GPCP}(\Lambda,\bm a,\Theta,\bm b)$.

{\bf Part} (i). Suppose, on the contrary, that $r(\bm x)$ is not a local Lipschitzian error bound for ${\rm GPCP}(\Lambda,\bm a,\Theta,\bm b)$. It then follows from Definition \ref{ICPD51} that there is a sequence $\{{\bm x}^{(i)}\}\subset\mathbb{R}^n$ satisfying $r({\bm x}^{(i)})\leq\varepsilon$ such that
\begin{align}\label{ICP511}
\frac{r({\bm x}^{(i)})}{{\rm dist}({\bm x}^{(i)},\Omega)}
=\frac{\left\|\min\left\{\sum_{k=1}^{m-1}{\mathcal A}^{(k)}({\bm x}^{(i)})^{m-k}+{\bm a},\sum_{p=1}^{l-1}{\mathcal B}^{(p)}({\bm x}^{(i)})^{l-p}+{\bm b}\right\}\right\|}{{\rm dist}({\bm x}^{(i)},\Omega)}\rightarrow0
\end{align}
as $ i\rightarrow\infty$. We now show that the sequence $\{{\bm x}^{(i)}\}$ is bounded. Assume that $\{\|{\bm x}^{(i)}\|\}\rightarrow\infty$ as $i\rightarrow\infty$, we will derive a contradiction. Obviously, it follows that either $\{{\bm x}^{(i)}\}$ satisfies \eqref{ICP54} or $\{{\bm x}^{(i)}\}$ does not satisfy \eqref{ICP54}. Therefore, we consider the following two cases.

{\bf Case} (a). If $\{{\bm x}^{(i)}\}$ satisfies \eqref{ICP54}, it follows from Condition \ref{ICPC53} that there is an  index $i_0\in[n]$ such that
$$\limsup_{i\rightarrow\infty}\frac{\min\left\{\left(\sum_{k=1}^{m-1}{\mathcal A}^{(k)}({\bm x}^{(i)})^{m-k}+{\bm a}\right)_{i_0},\left(\sum_{p=1}^{l-1}{\mathcal B}^{(p)}({\bm x}^{(i)})^{l-p}+{\bm b}\right)_{i_0}\right\}}{\|{\bm x}^{(i)}\|}> 0,$$
which implies that exists a subsequence $\{{\bm x}^{(i_j)}\}$ of $\{\bm x^{(i)}\}$ such that
$$
\min\left\{\left(\sum_{k=1}^{m-1}{\mathcal A}^{(k)}({\bm x}^{(i_j)})^{m-k}+{\bm a}\right)_{i_0},\left(\sum_{p=1}^{l-1}{\mathcal B}^{(p)}({\bm x}^{(i_j)})^{l-p}+{\bm b}\right)_{i_0}\right\}\rightarrow\infty
$$
as $i_j\rightarrow\infty$. By the definition of $r(\bm x)$ we further obtain that $r({\bm x}^{(i_j)})\rightarrow\infty$ as $i_j\rightarrow\infty$, which is a contradiction with $r(\bm x^{(i_j)})\leq \varepsilon$.

{\bf Case} (b). If $\{{\bm x}^{(i)}\}$ does not satisfy \eqref{ICP54}, then we have either
\begin{equation}\label{ICP513}
\frac{\left[-\left(\sum_{k=1}^{m-1}{\mathcal A}^{(k)}({\bm x}^{(i)})^{m-k}+{\bm a}\right)\right]_+}{\|{\bm x}^{(i)}\|}\nrightarrow\bm 0,
\end{equation}
or
\begin{equation}\label{ICP512}
\frac{\left[-\left(\sum_{p=1}^{l-1}{\mathcal B}^{(p)}({\bm x}^{(i)})^{l-p}+{\bm b}\right)\right]_+}{\|{\bm x}^{(i)}\|}\nrightarrow\bm 0.
\end{equation}
When (\ref{ICP513}) holds, there exist an index $j_0\in[n]$ and a subsequence $\{\bm x^{(i_j)}\}$ of $\{\bm x^{(i)}\}$, such that
$$\left(\sum_{k=1}^{m-1}{\mathcal A}^{(k)}({\bm x}^{(i_j)})^{m-k}+{\bm a}\right)_{j_0}\rightarrow-\infty \quad {\rm as}\quad i_j\rightarrow\infty.$$
Consequently, we have
$$
\min\left\{\left(\sum_{k=1}^{m-1}{\mathcal A}^{(k)}({\bm x}^{(i_j)})^{m-k}+{\bm a}\right)_{j_0},\left(\sum_{p=1}^{l-1}{\mathcal B}^{(p)}({\bm x}^{(i_j)})^{l-p}+{\bm b}\right)_{j_0}\right\}\rightarrow-\infty
$$
and hence $r({\bm x}^{(i_j)})\rightarrow\infty$ as $i_j\rightarrow\infty$, which is a contradiction. Similarly, when (\ref{ICP512}) holds, there must be an index $j_0\in[n]$ and a subsequence $\{\bm x^{(i_s)}\}\subset\{\bm x^{(i)}\}$ such that $\left(\sum_{p=1}^{l-1}{\mathcal B}^{(p)}({\bm x}^{(i_s)})^{l-p}+{\bm b}\right)_{j_0}\rightarrow-\infty$ as $i_s\rightarrow\infty$. Hence we obtain that $r({\bm x}^{(i_s)})\rightarrow\infty$ as $i_s\rightarrow\infty$, which is also a contradiction.

Both {\bf Cases} (a) and (b) indicate that the sequence $\{{\bm x}^{(i)}\}$ is bounded. Without loss of generality, we assume $\bm x^{(i)}\rightarrow \bar{\bm x}$ as $i\rightarrow\infty$. From \eqref{ICP511}, we have $$
r(\bm x^{(i)})=\left\|\min\left\{\sum_{k=1}^{m-1}{\mathcal A}^{(k)}({\bm x}^{(i)})^{m-k}+{\bm a},\sum_{p=1}^{l-1}{\mathcal B}^{(p)}({\bm x}^{(i)})^{l-p}+{\bm b}\right\}\right\|\rightarrow0
$$
as $i\rightarrow\infty$. Thus, $r(\bar{\bm x})=0$, and hence $\bar{\bm x}$ is a solution of ${\rm GPCP}(\Lambda,\bm a,\Theta,\bm b)$, i.e., $\bar{\bm x}\in\Omega$. Furthermore, by Assumption \ref{ICPA51}, we know that there exists a subsequence $\{\bm x^{(i_j)}\}$ of $\{\bm x^{(i)}\}$ such that
$$
\frac{r({\bm x}^{(i_j)})}{{\rm dist}({\bm x}^{(i_j)},\Omega)}\nrightarrow0~~{\rm as}~~i_j\rightarrow\infty,
$$
which contradicts \eqref{ICP511}. This indicates that $r(\bm x)$ is a local Lipschitzian error
bound for ${\rm GPCP}(\Lambda,\bm a,\Theta,\bm b)$.

{\bf Part} (ii). In terms of the result of {\bf Part} (i), we now show that $r(\bm x)$ is a global Lipschitzian error bound for ${\rm GPCP}(\Lambda,\bm a,\Theta,\bm b)$. Suppose, on the contrary, that the result
does not hold. Then it follows from Definition \ref{ICPD51} that for each positive integer $i$, there
exists an ${\bm x}^{(i)}\in \mathbb{R}^n$ such that
\begin{equation}\label{ICP514}
\|{\bm x}^{(i)}-\bar{\bm x}\|\geq {\rm dist}({\bm x}^{(i)},\Omega)> i\cdot r({\bm x}^{(i)})
\end{equation}
where $\bar{\bm x}$ is a fixed solution of ${\rm GPCP}(\Lambda,\bm a,\Theta,\bm b)$. Since $r(\bm x)$ is a local Lipschitzian error bound for the problem ${\rm GPCP}(\Lambda,\bm a,\Theta,\bm b)$ by {\bf Part} (i), there exists a positive integer $i_0$ and a constant $\varepsilon>0$ such that $r({\bm x}^{(i)})>\varepsilon$ for all $i>i_0$, which, together with \eqref{ICP514}, implies that $
\|{\bm x}^{(i)}\|\rightarrow\infty$ as $i\rightarrow\infty$. Obviously, we have either $\{{\bm x}^{(i)}\}$ satisfies \eqref{ICP54} or $\{{\bm x}^{(i)}\}$ does not satisfy \eqref{ICP54}.  If $\{{\bm x}^{(i)}\}$ satisfies \eqref{ICP54}, then \eqref{ICP554} holds by Condition \ref{ICPC53}. Consequently, by the definition of $r(\bm x)$, we have
\begin{equation}\label{ICP517}
\limsup_{i\rightarrow\infty}\frac{r({\bm x}^{(i)})}{\|{\bm x}^{(i)}\|}>0.
\end{equation}
If $\{{\bm x}^{(i)}\}$ does not satisfy \eqref{ICP54}, then since $\|{\bm x}^{(i)}\|\rightarrow \infty$ as $i\rightarrow\infty$, we have either \eqref{ICP513} or \eqref{ICP512} holds. If (\ref{ICP513}) holds, then there exist positive number $\bar \varepsilon$, an index $j_0\in [n]$ and a subsequence $\{{\bm x}^{(i_j)}\}$ of $\{\bm x^{(i)}\}$ such that
$$
\left(\sum_{k=1}^{m-1}{\mathcal A}^{(k)}({\bm x}^{(i_j)})^{m-k}+{\bm a}\right)_{j_0}\leq-\bar \varepsilon\|{\bm x}^{(i_j)}\|
$$
for every $i_j$, which implies $r(\bm x^{(i_j)})\geq \bar \varepsilon\|\bm x^{(i_j)}\|$ for every $i_j$. Hence, (\ref{ICP517}) holds. Similarly, when (\ref{ICP512}) holds, we also can claim that (\ref{ICP517}) holds.

Take any given positive number $M$. Considering \eqref{ICP514} with $i\geq M$ and \eqref{ICP517}, by letting $i\rightarrow \infty$, we have
$$
1=\lim_{i\rightarrow\infty}\frac{\|{\bm x}^{(i)}-\bar{\bm x}\|}{\|{\bm x}^{(i)}\|}\geq\limsup_{i\rightarrow\infty}~ i\frac{r({\bm x}^{(i)})}{\|{\bm x}^{(i)}\|}\geq M\limsup_{i\rightarrow\infty}\frac{r({\bm x}^{(i)})}{\|{\bm x}^{(i)}\|},
$$
which means $1\geq +\infty$ from the arbitrariness of $M$. A contradiction yields and the proof is complete.
\end{proof}

From Theorem \ref{theorem1.6}, Lemma \ref{ICPlemma}, and Theorem \ref{ICPI51}, we immediately obtain the following global Lipschitzian error bound.
\begin{theorem}
Let $\Lambda=\left({\mathcal A}^{(1)},\ldots,{\mathcal A}^{(m-1)}\right)\in \mathcal{F}_{m,n}$, $\Theta=\left({\mathcal B}^{(1)},\ldots,{\mathcal B}^{(l-1)}\right)\in \mathcal{F}_{l,n}$. Suppose that the leading tensor pair $({\mathcal A}^{(1)},{\mathcal B}^{(1)})$ in $(\Lambda,\Theta)$ is an ${\bf ER}$-tensor pair.  If either (a) $m=l$ or (b) $m\neq l$, but the one of $\mathcal{A}^{(1)}$ and $\mathcal{B}^{(1)}$, whose order is even and is larger than another, is positive definite, then for any given vectors $\bm a, \bm b\in \mathbb{R}^n$ such that Assumption \ref{ICPA51} holds, $r(\bm x)$ is a global Lipschitzian error bound for ${\rm GPCP}(\Lambda,\bm a,\Theta,\bm b)$.
\end{theorem}

Similarly, based upon Remarks \ref{Rem1} and \ref{Rem2}, we obtain the following theorem.
\begin{theorem}
Let $\Lambda=\left({\mathcal A}^{(1)},\ldots,{\mathcal A}^{(m-1)}\right)\in \mathcal{F}_{m,n}$, $\Theta=\left({\mathcal B}^{(1)},\ldots,{\mathcal B}^{(l-1)}\right)\in \mathcal{F}_{l,n}$. Suppose that the leading tensor pair $({\mathcal A}^{(1)},{\mathcal B}^{(1)})$ in $(\Lambda,\Theta)$ is an ${\bf ER}$-tensor pair. If $m>l$ and $\mathcal{A}^{(1)}$ is strictly copositive (see \cite{Qi13}), then for any given vectors $\bm a, \bm b\in \mathbb{R}^n$ satisfying Assumption \ref{ICPA51} and the function $F$ in \eqref{gpcp} satisfying $S(\mathbb{R}_+^n)\subset \mathbb{R}_+^n$, where $S$ is same to that in Remark \ref{Rem2}, $r(\bm x)$ is a global Lipschitzian error bound for ${\rm GPCP}(\Lambda,\bm a,\Theta,\bm b)$.
\end{theorem}

Now we complete this section by considering the TCPs introduced in \cite{SQ16}, i.e., the model ${\rm TCP}({\mathcal A},{\bm q})$. In this situation, we may prove that, for a given $\mathcal{A}\in \mathbb{T}_{m,n}$ and $\bm q\in \mathbb{R}^n$, if $\mathcal{A}$ is an ${\bf R}_0$-tensor, then Condition \ref{ICPC53} holds. Furthermore, by \cite[Theorem 4.2]{WHB16}, we know that if $\mathcal{A}\in \mathbb{T}_{m,n}$ is an ${\bf ER}$-tensor and $\bm q\in \mathbb{R}^n$ is given, then the solution set of ${\rm TCP}(\mathcal{A},\bm q)$ is nonempty and compact. By \cite[Theorem 3.2]{SQ14}, we know that if $\mathcal{A}\in \mathbb{T}_{m,n}$ is an ${\bf R}$-tensor (see \cite[Definition 2.2]{SQ14}) and $\bm q\in \mathbb{R}^n$ is given, then the solution set of ${\rm TCP}(\mathcal{A},\bm q)$ is nonempty. Therefore, by Theorem \ref{ICPI51}, we have the following result.

\begin{theorem}\label{TCPGEB}
Given $\mathcal{A}\in \mathbb{T}_{m,n}$ and $\bm q\in \mathbb{R}^n$. Suppose that $\mathcal{A}$ is an ${\bf ER}$-tensor (or {\bf R}-tensor). If Assumption \ref{ICPA51} with $F(\bm x)=\mathcal{A}\bm x^{m-1}+\bm q$ and $G(\bm x)=\bm x$ holds, then $r(\bm x)$ is a global Lipschitzian error bound for ${\rm TCP}(\mathcal{A},\bm q)$.
\end{theorem}

Note that the result in Theorem \ref{TCPGEB}, to our knowledge, is not discussed in the current TCPs and PCPs literature (e.g., see \cite{G16,LHL17,YLH17}). On the other hand, our Assumption \ref{ICPA51} is weaker than the conditions assumed in previous papers. Therefore, our result is better.

\section{Conclusion}\label{FinRem}
The GPCP under consideration is a natural generalization of TCPs and PCPs, and a special case of GCPs, but with more favorable polynomial nature that we could explore to derive interesting specialized results than the general nonlinear functions in GCPs. In this paper, we obtain some new results on GPCPs, which include the nonemptiness and compactness of the solution set, basic topological properties, and global Lipschitzian error bounds of solutions of GPCP($\Lambda,\bm a,\Theta,\bm b$) with structured (e.g., {\bf ER}-) tensor pair. In the future, we will pay attention to designing some structure-exploiting algorithms for GPCPs.

\bigskip
\noindent{\bf Acknowledgements} The authors would like to thank the two reviewers' close reading and valuable comments, which helped us improve the presentation of this paper. L. Ling was supported by Excellent Degree Thesis Foundation of HDU (No. yxlw2017018). C. Ling and H. He were supported in part by National Natural Science Foundation of China (Nos. 11571087 and 11771113) and Natural Science Foundation of Zhejiang Province (LY17A010028).


\begin{thebibliography}{10}
\expandafter\ifx\csname urlstyle\endcsname\relax
  \providecommand{\doi}[1]{doi:\discretionary{}{}{}#1}\else
  \providecommand{\doi}{doi:\discretionary{}{}{}\begingroup
  \urlstyle{rm}\Url}\fi

\bibitem{BHW16}
X.L. Bai, Z.H. Huang and Y.~Wang, Global uniqueness and solvability for tensor
  complementarity problems, \textit{J. Optim. Theory Appl.} 170 (2016), 72--84.

\bibitem{BACGWM15}
H.~Becker, L.~Albera, P.~Comon, R.~Gribonval, F.~Wendling and I.~Merlet,
  Brain-source imaging: From sparse to tensor models, \textit{IEEE Signal
  Process Mag.} 32 (2015), 100--112.

\bibitem{Be93}
G.~Beer, \textit{Topologies on Closed and Closed Convex Sets}, Kluwer Academic
  Publishers, Dordrecht, The Netherlands, 1993.

\bibitem{BV08}
L.~Bloy and R.~Verma, On computing the underlying fiber directions from the
  diffusion orientation distribution function, in: D.~Metaxas, L.~Axel,
  G.~Fichtinger and G.~Sz\'{e}kely, eds., \textit{Medical Image Computing and
  Computer-Assisted Intervention--MICCAI 2008}, pp. 1--8, Springer,
  Berlin/Heidelberg, 2008.

\bibitem{CQW16}
M.L. Che, L.Q. Qi and Y.M. Wei, Positive definite tensors to nonlinear
  complementarity problems, \textit{J. Optim. Theory Appl.} 168 (2016),
  475--487.

\bibitem{CCLQ18}
H.~Chen, Y.~Chen, G.~Li and L.Q. Qi, A semidefinite program approach for
  computing the maximum eigenvalue of a class of structured tensors and its
  applications in hypergraphs and copositivity test, \textit{Numer. Linear
  Algebra Appl.} 25 (2018), e2125.

\bibitem{CPS92}
R.W. Cottle, J.~S. Pang and R.~E. Stone, \textit{The Linear Complementarity
  Problem}, Academic Press, Boston, 1992.

\bibitem{DQW13}
W.~Ding, L.Q. Qi and Y.M. Wei, M-tensor and nonsingular {M}-tensors,
  \textit{Linear Algebra Appl.} 439 (2013), 3264--3278.

\bibitem{FP03}
F.~Facchinei and J.S. Pang, \textit{Finite-Dimensional Variational Inequalities
  and Complementarity Problems}, Springer, New York, 2003.

\bibitem{FP97}
M.C. Ferris and J.S. Pang, Engineering and economic applications of
  complementarity problems, \textit{SIAM Rev.} 39 (1997), 669--713.

\bibitem{G16}
M.~S. Gowda, Polynomial complementarity problems, \textit{Pac. J. Optim.} 13
  (2017), 227--241.

\bibitem{GLQX15}
M.S. Gowda, Z.Y. Luo, L.Q. Qi and N.H. Xiu, Z-tensors and complementarity
  problems, \textit{arXiv: 1510.07933}  (2015).

\bibitem{Hua05}
Z.H. Huang, Global {L}ipschitzian error bounds for semidefinite complementarity
  problems with emphasis on {NCP}s, \textit{Appl. Math. Comput.} 162 (2005),
  1237--1258.

\bibitem{HQ16}
Z.H. Huang and L.Q. Qi, Formulating an $n$-person noncooperative game as a
  tensor complementarity problem, \textit{Comput. Optim. Appl.} 66 (2017),
  557--576.

\bibitem{IBK97}
G.~Isac, V.~Bulavski and V.~Kalashnikov, Exceptional families, topological
  degree and complementarity problems, \textit{J. Global Optim.} 10 (1997),
  207--225.

\bibitem{IC99}
G.~Isac and A.~Carbone, Exceptional families of elements for continuous
  functions: some applications to complementarity theory, \textit{J. Global
  Optim.} 15 (1999), 181--196.

\bibitem{KI02}
Vyacheslav~V Kalashnikov and George Isac, Solvability of implicit
  complementarity problems, \textit{Ann. Oper. Res.} 116 (2002), 199--221.

\bibitem{KF96}
C.~Kanzow and M.~Fukushima, Equivalence of the generalized complementarity
  problem to differentiable unconstrained minimization, \textit{J. Optim.
  Theory Appl.} 90 (1996), 581--603.

\bibitem{KS80}
D.~Kinderlehrer and G.~Stampacchia, \textit{An Introduction to Variational
  Inequalities and Their Applications}, Academic Press, New York, 1980.

\bibitem{LHQ16}
C.~Ling, H.J. He and L.Q. Qi, Higher-degree eigenvalue complementarity problems
  for tensors, \textit{Comput. Optim. Appl.} 64 (2016), 149--176.

\bibitem{LHL17}
L.Y. Ling, H.J. He and C.~Ling, On error bounds of polynomial complementarity
  problems with structured tensors, \textit{Optimization} 67 (2018), 341--358.

\bibitem{L78}
N.~Lloyd, \textit{Degree Theory}, Cambridge University Press, Cambridge, U.K.,
  1978.

\bibitem{NQZ09}
M.~Ng, L.Q.~Qi and G.L. Zhou, Finding the largest eigenvalue of a non-negative
  tensor, \textit{SIAM J. Matrix Anal. Appl.} 31 (2009), 1090--1099.

\bibitem{Qi2005}
L.Q. Qi, Eigenvalues of a real supersymmetric tensor, \textit{J. Symbolic
  Comput.} 40 (2005), 1302--1324.

\bibitem{Qi13}
L.Q. Qi, Symmetric nonnegative tensors and copositive tensors, \textit{Linear
  Algebra Appl.} 439 (2013), 228--238.

\bibitem{QWW09}
L.Q. Qi, F.~Wang and Y.~Wang, Z-eigenvalue methods for a global polynomial
  optimization problem, \textit{Math. Program.} 118 (2009), 301--316.

\bibitem{QYW10}
L.Q. Qi, G.H. Yu and Ed~X. Wu, Higher order positive semidefinite diffusion
  tensor imaging, \textit{SIAM J. Imaging Sci.} 3 (2010), 416--433.

\bibitem{RW98}
R.~T. Rockafellar and R.~J.~B. Wets, \textit{Variational Analysis},
  Springer-Verlag, Berlin, 1998.

\bibitem{SQ14}
Y.S.~Song and L.Q.~Qi, An initial study on ${P}$, ${P}_0$, ${B}$ and ${B}_0$
  tensors, ~ArXiv:1403.1118v3, 2015.

\bibitem{SQ15a}
Y.S. Song and L.Q. Qi, Eigenvalues and structured properties of {P}-tensors,
  \textit{arXiv:1508.02005v3}  (2015).

\bibitem{SQ15b}
Y.S. Song and L.Q. Qi, Properties of some classes of structured tensors,
  \textit{J. Optim. Theory Appl.} 165 (2015)(3), 854--873.

\bibitem{SQ16}
Y.S. Song and L.Q. Qi, Tensor complementarity problem and semi-positive
  tensors, \textit{J. Optim. Theory Appl.} 169 (2016), 1069--1078.

\bibitem{WW67}
D.~W. Walkup and R.~J.~B. Wets, Continuity of some convex-cone-valued mappings,
  \textit{Proc. Amer. Math. Soc.} 18 (1967), 229--235.

\bibitem{Wal13}
L.~Walras, \textit{Elements of Pure Economics}, Allen and Unwin, London, 1954.

\bibitem{WHB16}
Y.~Wang, Z.H. Huang and X.L. Bai, Exceptionally regular tensors and tensor
  complementarity problems, \textit{Optim. Method Softw.} 31 (2016), 815--828.

\bibitem{XZ02}
N.H. Xiu and J.Z. Zhang, Global projection-type error bounds for general
  variational inequalities, \textit{J. Optim. Theory Appl.} 112 (2002),
  213--228.

\bibitem{YLH17}
W.~Yu, C.~Ling and H.J. He, On the properties of tensor complementarity
  problems, \textit{Pac. J. Optim.}  (2017), to appear.

\bibitem{Zha07}
L.P. Zhang, A nonlinear complementarity model for supply chain network
  equilibrium, \textit{J. Ind. Manag. Optim.} 3 (2007), 727--737.

\bibitem{ZZZXC16}
G.~Zhou, Q.~Zhao, Y.~Zhang, T.~Adal{\i}, S.~Xie and A.~Cichocki, Linked
  component analysis from matrices to high-order tensors: Applications to
  biomedical data, \textit{Proc. IEEE} 104 (2016), 310--331.

\end{thebibliography}

\hbox to14cm{\hrulefill}\par

\begin{flushright}
{\it Manuscript received 31 May 2018} \\ \medskip
{\it Revised 16 September 2018} \\  \medskip
{\it Accepted for publication 1 November 2018}
\end{flushright}
\bigskip

\noindent {\sc Liyun Ling}\\
Department of Mathematics, School of Science, Hangzhou Dianzi University,
Hangzhou, 310018, China.\\
E-mail address: {lingliyun@163.com} \\ \bigskip

\noindent {\sc Chen Ling}\\
Department of Mathematics, School of Science, Hangzhou Dianzi University,
Hangzhou, 310018, China.\\
E-mail address: {macling@hdu.edu.cn} \\
\bigskip

\noindent {\sc Hongjin He}\\
Department of Mathematics, School of Science, Hangzhou Dianzi University,
Hangzhou, 310018, China.\\
Corresponding author.\\
E-mail address: {hehjmath@hdu.edu.cn}

\end{document}